\documentclass[a4paper]{article}%
\usepackage{amsfonts}
\usepackage{amsmath}
\usepackage{amssymb}
\usepackage{graphicx}
\usepackage{color}
\usepackage{bibmods}
\usepackage{accents}
\usepackage{hyperref}
\newtheorem{theorem}{Theorem}

\newtheorem{definition}[theorem]{Definition}

\newtheorem{lemma}[theorem]{Lemma}

\newtheorem{proposition}[theorem]{Proposition}
\newtheorem{remark}[theorem]{Remark}

\newenvironment{proof}[1][Proof]{\textit{#1.} }{\  \rule{0.5em}{0.5em}}

\begin{document}

\title{On the existence of solutions to the \\planar exterior Navier Stokes system}
\author{Matthieu Hillairet\\{\small Universit\'{e} Paris Dauphine }\\{\small Place du Mar\'{e}chal De Lattre De Tassigny}\\{\small 75775 Paris Cedex 16 - France}\\{\small hillairet@ceremade.dauphine.fr}
\and Peter Wittwer\thanks{Work supported in part by the Swiss National Science
Foundation.}\\{\small University of Geneva}\\{\small 24, Quai Ernest Ansermet}\\{\small 1205 Geneva - Switzerland}\\{\small \ peter.wittwer@unige.ch }}
\date{\today }
\maketitle

\begin{abstract}
We consider the stationary incompressible Navier Stokes equation in the
exterior of a disk $B\subset \mathbb{R}^{2}$ with non-zero Dirichlet boundary
conditions on the disk and zero boundary conditions at infinity. We prove the
existence of solutions for an open set of boundary conditions without symmetry.

\end{abstract}
\tableofcontents

\section{Introduction}

In this paper we consider the incompressible Navier Stokes equations in an
exterior domain:
\begin{equation}
\left \{
\begin{array}
[c]{r}%
\Delta \mathbf{u}-\nabla p=\mathbf{u}\cdot \nabla \mathbf{u}\text{ },\\
\operatorname{div}\mathbf{u}=0\text{ },
\end{array}
\right.  \text{ in $\mathbb{R}^{2}\setminus \overline{B}$ }, \label{NS}%
\end{equation}
with $B$ a smooth bounded domain, with non-zero Dirichlet boundary conditions
on $\partial B$, and zero boundary conditions at infinity:%
\begin{equation}
\mathbf{u}_{|_{\partial B}}=\mathbf{u}^{\ast}\text{ },\quad \lim_{|\mathbf{x}%
|\rightarrow \infty}\mathbf{u}(\mathbf{x})=0\text{ }. \label{BC}%
\end{equation}
Of particular interest is the case of boundary data $\mathbf{u}^{\ast}$ with
zero flux:
\begin{equation}
\int_{\partial B}\mathbf{u}^{\ast}\cdot \mathbf{n}\text{ }\mathrm{d}%
\sigma=0\text{ }. \label{fluxnul}%
\end{equation}
We note that, since the size of $B$ is arbitrary, we have set without
restriction of generality all the physical constants in (\ref{NS}) equal to one.

The above system is a special case of the exterior Navier Stokes problem:
\begin{equation}
\left \{
\begin{array}
[c]{r}%
-\left(  \mathbf{u}\cdot \nabla \right)  \mathbf{u}-\lambda \partial
_{1}\mathbf{u}+\Delta \mathbf{u}-\nabla p=0\text{ },\\
\nabla \cdot \mathbf{u}=0\text{ },
\end{array}
\right.  \text{ in }\mathbb{R}^{n}\setminus \overline{B}\text{ },
\label{eq_NSE3}%
\end{equation}
with $n=2$ or $3$, with $B$ a smooth bounded domain in $\mathbb{R}^{n}$, with
boundary conditions (\ref{BC}), and with $\lambda \in \{0,1\}$ distinguishing
between the case of a flow \textquotedblleft around\textquotedblright \ $B$
($\lambda=0$) and a flow \textquotedblleft past\textquotedblright \ $B$
($\lambda=1$), respectively. The system (\ref{NS})-(\ref{fluxnul}) corresponds
to $n=2$ and $\lambda=0$. The case $\lambda=0$ is in many respects more
complicated than the case $\lambda=1$, and, whereas the picture is rather
complete for $n=3$, the case $n=2$, $\lambda=0$, presents particular
difficulties. The difficulty with the classical method for solving the Navier Stokes equations consists in the fact that the linearization around $\mathbf{u=0}$ is given by the Stokes
system, which, for $n=2$, does not admit a solution satisfying (\ref{BC}),
unless the domain $B$ and the boundary data $\mathbf{u}^{\ast}$ satisfy
certain symmetry conditions. This fact is known as the Stokes paradox. {For
completeness we note that if one relaxes the no flux condition (\ref{fluxnul}%
), there exists a two parameter family of solutions to (\ref{NS}%
)-(\ref{fluxnul}), the so called Hamel solutions, see \cite{Galdi94II}. {These
examples emphasize that the decay of solutions can be arbitrary slow and that
uniqueness might be lost for some boundary data. However, these solutions have
flux larger than one, and are far from the regime which we will consider
here.}\medskip}

In what follows we construct a new class of solutions to (\ref{NS}%
)-(\ref{fluxnul}), by linearizing not around $\mathbf{u=0}$, but around
$\mathbf{u}=\mu \mathbf{x}^{\perp}/\left \vert \mathbf{x}\right \vert ^{2}$, with
$\left \vert \mu \right \vert > \sqrt{48}$. This improves the decay of the
solutions to the vorticity equation, yielding vorticities decaying at infinity
generically faster than $\left \vert \mathbf{x}\right \vert ^{-2}$, instead of
like $\left \vert \mathbf{x}\right \vert ^{-1}$ as would be the case for the
Stokes equation, thus avoiding the Stokes paradox when reconstructing
$\mathbf{u}$ \textit{via} the Bio-Savart law.

\medskip

To put our problem into a wider context, we briefly recall the concept of weak
solutions for (\ref{eq_NSE3}), (\ref{BC}) { (also known as generalized
solutions or $D$--solutions),} and the method of J. Leray \cite{Leray33} for
proving the existence of such weak solutions.

\begin{definition}
Given $\mathbf{u}^{\ast}\in H^{1/2}(\partial{B})$ satisfying (\ref{fluxnul}),
a function $\mathbf{u}$ which satisfies the following conditions is called a
weak solution to (\ref{eq_NSE3}), (\ref{BC}):

\begin{enumerate}
\item $\mathbf{u}\in D^{1,2}(\mathbb{R}^{n}\setminus \overline{B})$, where
$D^{1,2}(\mathbb{R}^{n}\setminus \overline{B})$ is the subset of $L_{loc}%
^{1}(\mathbb{R}^{2}\setminus \overline{B})$ containing functions with gradient
in $L^{2}(\mathbb{R}^{n}\setminus \overline{B}),$

\item $\mathbf{u}$ is divergence-free and $\mathbf{u}=\mathbf{u}^{\ast}$ on
$\partial{B}$,

\item for all divergence-free vector fields $\mathbf{w}\in 
C_{c}^{\infty}(\mathbb{R}^{n}\setminus B)$, {there holds}:%
\[
\int_{\mathbb{R}^{n}\setminus \overline{B}}\nabla \mathbf{u}\colon
\nabla \mathbf{w} + \int_{\mathbb{R}^{n}\setminus \overline{B}}(\left(
\mathbf{u}\cdot \nabla \right)  \mathbf{u}+\lambda \partial_{1}\mathbf{u}%
)\cdot \mathbf{w}=0\text{ }.
\]

\end{enumerate}
\end{definition}

The method of J. Leray to prove the existence of solutions according to this
definition, and a posteriori to (\ref{eq_NSE3}), (\ref{BC}), in the sense of
distributions, consists in the following steps:

\begin{itemize}
\item First, one introduces a sequence of approximate problems by restricting
(\ref{eq_NSE3}) to bounded subsets $\Omega \subset \mathbb{R}^{n} $ containing
$B$, with zero Dirichlet boundary conditions on $\partial \Omega \setminus
\partial{B}$.

\item Second, one proves the existence of (weak) solutions to all these
approximate problems.

\item Third, one shows that for any sequence of bounded subsets exhausting
$\mathbb{R}^{n}\setminus \overline{B}$, there exists a subsequence, such that
the corresponding approximate solutions converge to a weak solution of
(\ref{eq_NSE3}), (\ref{BC}).

\item Finally, given a weak solution $\mathbf{u}$, a pressure $p$ an be
constructed \textit{via} De Rham's theory, such that the equations
(\ref{eq_NSE3}) are satisfied in $\mathcal{D}^{\prime}(\mathbb{R}^{n}%
\setminus \overline{B})$.
\end{itemize}

 See also
\cite{Hillairet07d,HillairetSerre03,Serre87,Weinberger72,Weinberger73}, where
this method has been adapted {to a similar system with more general boundary
conditions}.  Note that if $B$ has a smooth boundary, the ellipticity of the Stokes operator
(see \cite[Section IX.1]{Galdi94I}) and the smoothness of $\mathbf{u}^{\ast}$
imply that weak solutions are smooth. Therefore, for smooth data, the only possible
shortcoming of weak solutions is that they may not satisfy the boundary
condition at infinity in a point-wise sense. {Much} work has been devoted to
clarify the situation in various cases (see \cite{Galdi94II} for more details):

\bigskip

For $n=3$, the condition $\mathbf{u}\in D^{1,2}(\mathbb{R}^{3}\setminus
\overline{B})$ implies that weak solutions tend to zero at infinity.
 The exact decay can be obtained by various methods yielding the following results:

\begin{itemize}
\item for $\lambda=1$, there exists a solution that decays like the
fundamental solution of the Oseen equation (the linear system obtained from
(\ref{eq_NSE3}) by deleting the nonlinear convective terms)
\cite{Babenko73,Farwig92,Finn59}. This result can be obtain by a detailed
analysis of the Oseen equation with a source term in the usual Sobolev spaces
\cite{Babenko73,Finn59}, and also in weighted Sobolev spaces \cite{Farwig92}.

\item for $\lambda=0$ and sufficiently small boundary data, there exists a
unique weak solution, and this solution decays like a Landau solution
\cite{KorolevSverak10}, a special solution of the nonlinear system which
decays like $1/|\mathbf{x}|.$ This result can been obtained by constructing
first a strong solution to (\ref{BC}), (\ref{eq_NSE3}), which is asymptotic to
the Landau solution, by perturbative techniques. Using the known decay of this
particular solution as an input \cite[Section IX.9]{Galdi94II}, one then
proves a weak-strong uniqueness result for small data.
\end{itemize}

For $n=2$, the situation is more delicate since the condition $\mathbf{u}\in
D^{1,2}(\mathbb{R}^{2}\setminus \overline{B})$ does not guarantee that the
boundary condition at infinity is satisfied:

\begin{itemize}
\item For $\lambda=1$, the relevant linear system is again the Oseen equation,
but the results concerning the decay are limited to small data, since, as for
the case $n=3$, $\lambda=0$, perturbative techniques are used to prove the
existence of a strong solution decaying at infinity like the fundamental
solution of the Oseen equation. This solution is then again used as an input
to a weak strong uniqueness argument in order to show the decay of weak
solutions. These results can be found in \cite{Galdi93}.

\item The case $\lambda=0$ remains largely open. {As we already pointed out,
the} problem is that the solution to the Stokes equation with boundary data
$\mathbf{u}^{\ast}\neq0$ diverges at infinity, unless one makes additional
assumptions on the domain $B$ and the data $\mathbf{u}^{\ast}$. 
Partial results for the Navier Stokes system {with symmetric data} can be found in 
\cite{Galdi02,Russo10,Russopp,PileckasRusso12} .
\end{itemize}

\bigskip

From now on we limit the discussion to the case where $B$ is a disk of radius
one. We choose $\mathbf{x}=(x,y)$ Cartesian coordinates with the origin at the
center of $B$, $(r,\theta)\in \Omega:=(0,\infty)\times(-\pi,\pi)$ the
associated polar coordinates, and $(\mathbf{e}_{r},\mathbf{e}_{\theta})$ the
corresponding local orthonormal basis. For the function $\mathbf{u}$ we have
in polar coordinates:
\begin{equation}
\mathbf{u}(r,\theta)=u_{r}(r,\theta)\mathbf{e}_{r}+u_{\theta}(r,\theta
)\mathbf{e}_{\theta},\quad \forall \text{ }(r,\theta)\in \Omega \,. \label{pc}%
\end{equation}

The following theorem is our main result:

\begin{theorem}
\label{thm_main} Let $\mu_{0}>{\mu_{crit}}\equiv \sqrt{48}$ and $\mathbf{u}%
^{\ast}\in C^{\infty}(\partial B)$ satisfying \eqref{fluxnul} be
sufficiently close to ${\mathbf{u}_{\mu_{0}}^{\ast}:=}\mu_{0}\mathbf{e}%
_{\theta}$. Then, the equations (\ref{NS}), (\ref{BC}), with boundary
condition $\mathbf{u}^{\ast}$, have at least one solution $(\mathbf{u}%
,p)\in C^{\infty}(\mathbb{R}^{2}\setminus \overline{B})^{2}%
\times C^{\infty}(\mathbb{R}^{2}\setminus \overline{B})$. Moreover,
there exist $\mu$ close to $\mu_{0}$ such that:
\begin{equation}
\lim_{r\rightarrow \infty}\text{ }r\left \Vert \mathbf{u}(r,\theta)-\frac
{\mu \mathbf{e}_{\theta}}{r};L^{\infty}(-\pi,\pi)\right \Vert =0\text{ }.
\label{bound}%
\end{equation}
\end{theorem}

\begin{remark}
If the pair $(u(x,y),v(x,y))$ is a solution for the boundary condition
$(u^{\ast}(x,y),v^{\ast}(x,y))$, then the pair $(u(x,-y),-v(x,-y))$ is a
solution for the boundary condition $(u^{\ast}(x,-y),-v^{\ast}(x,-y))$. Thus,
our result extends to $\mu_{0}<-{\mu_{crit}}$.
\end{remark}

\begin{remark}
If $\mathbf{u}(r,\theta)$ is a solution for the boundary condition\textbf{\ }%
$\mathbf{u}^{\ast}$on the complement of the unit disk, then for all
$\lambda>0$, $\lambda \mathbf{u}(\lambda r,\theta)$ is a solution for the
boundary condition $\lambda \mathbf{u}^{\ast}$ on the complement of the disk of
radius $\lambda^{-1}$.
\end{remark}

\begin{remark}
The restriction to the case where $B$ is a disk is for the sake of simplicity
only. This permits to rewrite the system in polar coordinates, yielding
explicit expressions for the solutions. We expect that with more work the
results can be generalized to arbitrary smooth $B$.
\end{remark}

To prove Theorem \ref{thm_main} we proceed as follows: We fix $\mu>\mu_{crit}$
and consider the pair $(\mathbf{u}_{\mu},p_{\mu})$:
\begin{equation}
\mathbf{u}_{\mu}(r,\theta)=\dfrac{\mu \mathbf{e}_{\theta}}{r}\text{ },\qquad
p_{\mu}(r,\theta)=-\frac{1}{2}\dfrac{\mu^{2}}{r^{2}}\text{ },\qquad
\forall \,(r,\theta)\in{\Omega}\text{ }, \label{explicit}%
\end{equation}
which is an exact solution to (\ref{NS}), (\ref{BC}). Next we set,
$(\mathbf{u},p)=(\mathbf{u}_{\mu}+\mathbf{v},p=p_{\mu}+$ ${q})$ and prove,
that for all sufficiently small boundary conditions $\mathbf{v}^{\ast}$
satisfying%
\begin{equation}
\int_{\partial B}\mathbf{v}^{\ast}\cdot \mathbf{n}\text{ }\mathrm{d}%
\sigma=0\text{ },
\end{equation}
there existence of a solution $(\mathbf{v},$ ${q})\in C^{\infty
}(\mathbb{R}^{2}\setminus \overline{B})^{2}\times C^{\infty
}(\mathbb{R}^{2}\setminus \overline{B})$ such that $\left.  \mathbf{v}%
\right \vert _{\partial B}=\mathbf{v}^{\ast}+\mu_{\ast}$, for some $\mu_{\ast
}>{\mu_{crit}}$ depending on $\mu$ and $\mathbf{v}^{\ast}$. In a final step, we
show that this function can be inverted, giving $\mu$ as a function of
$\mu_{\ast}$ and $\mathbf{v}^{\ast}$, thus yielding Theorem \ref{thm_main}.

The feasibility of our approach relies on the fact {that }the system obtained
by linearizing (\ref{NS}), (\ref{BC}) around the explicit solution
$(\mathbf{u}_{\mu},p_{\mu})$ can be analyzed explicitly. As mentioned above,
when compared with the case $\mu=0$,\textit{\ i.e.}, the Stokes equation, the vorticity decays for $\mu>{\mu_{crit}}$
faster than ${1/r^{2}}$, instead of like ${1/r}$, such that $\mathbf{u}$ can be
shown {to} decay faster than $1/r$ at infinity, making the nonlinearity
subcritical. Introducing suitable function spaces, we are then able to solve
the full non-linear system by a classical fixed-point argument.

\section{Dynamical system formulation}

Let $(\mathbf{u},p)\in C^{\infty}(\mathbb{R}^{2}\setminus B)$ be a
solution to (\ref{NS}), (\ref{BC}), satisfying (\ref{fluxnul}). We first make
the construction of the stream-function $\psi$
associated with $\mathbf{u}$ precise. Since $\mathbf{u}\in C^{\infty
}(\mathbb{R}^{2}\setminus B)$, we have in particular that $\mathbf{u}^{\ast
}\in C^{\infty}(\partial B)$. Let $\mathbf{u}_{int} \in C%
^{\infty}(\overline{B})$ satisfy $\mathbf{u}_{int}=\mathbf{u}^{\ast}$ on
$\partial{B}$. Such a function exists since
$\mathbf{u}^{\ast}$ satisfies (\ref{fluxnul}). For instance, $\mathbf{u}%
_{int}$ can be the solution to the Stokes equations on $B$, with boundary
condition $\mathbf{u}^{\ast}$ on $\partial B$. Then, setting:
\[
{\mathbf{\bar{u}}}(x,y)=\left \{
\begin{array}
[c]{ll}%
\mathbf{u} & \text{in $\mathbb{R}^{2}\setminus \overline{B\,}$},\\
\mathbf{u}_{int} & \text{in $\overline{B}$}\,,
\end{array}
\right.
\]
we obtain a continuous divergence-free vector-field on the whole of
$\mathbb{R}^{2}$. Furthermore, this function is smooth on both sides of
$\partial B$ so that there exists $\psi \in C^{1}(\mathbb{R}^{2}%
)\cap C^{\infty}(\overline{B})\cap C^{\infty}(\mathbb{R}%
^{2}\setminus{B})$ satisfying $\mathbf{u}=\nabla^{\bot}\psi$.

Instead of (\ref{NS}), (\ref{BC}), we consider now the equation for the stream
function $\psi$ and the vorticity $\omega=\nabla \times \mathbf{u}$,
\begin{equation}
\left \{
\begin{array}
[c]{rcl}%
\Delta \psi & = & -\omega \\
\Delta \omega & = & \mathbf{u}\cdot \nabla \omega
\end{array}
\right.  \quad \text{ in $\mathbb{R}^{2}\setminus \overline{B}\,.$}\nonumber
\end{equation}
For the function $\mathbf{u}$ we have in polar coordinates (\ref{pc}), and the
vorticity becomes:%
\[
\omega=\dfrac{1}{r}\partial_{r}(ru_{\theta})-\dfrac{1}{r}\partial_{\theta
}u_{r},\quad \forall \text{ }(r,\theta)\in \Omega \,.
\]
For the boundary data we have:%
\[
\mathbf{u}^{\ast}(\theta)=u_{r}^{\ast}(\theta)\mathbf{e}_{r}+u_{\theta}^{\ast
}(\theta)\mathbf{e}_{\theta}\text{ },\quad \forall \text{ }\theta \in(-\pi
,\pi)\,.
\]
In polar coordinates we get the following equations for the stream function
$\psi$ and the vorticity $\omega$:
\begin{equation}
\left \{
\begin{array}
[c]{rcl}%
\partial_{rr}\psi+\dfrac{1}{r}\partial_{r}\psi+\dfrac{1}{r^{2}}\partial
_{\theta \theta}\psi & = & -\omega \,,\\[6pt]%
\partial_{rr}\omega+\dfrac{1}{r}\partial_{r}\omega+\dfrac{1}{r^{2}}%
\partial_{\theta \theta}\omega & = & u_{r}\partial_{r}\omega+\dfrac{u_{\theta}%
}{r}\partial_{\theta}\omega \,,
\end{array}
\right.  \quad \forall \text{ }(r,\theta)\in \Omega \text{ }, \label{NS_polar}%
\end{equation}
and%
\begin{equation}
\left \{
\begin{array}
[c]{rcl}%
\smallskip u_{r} & = & \dfrac{\partial_{\theta}\psi}{r}\,,\\
u_{\theta} & = & -\partial_{r}\psi \,,
\end{array}
\right.  \text{ }\forall \text{ }(r,\theta)\in \Omega \text{ },
\label{comp_polar}%
\end{equation}
together with the boundary conditions:%
\begin{equation}
\left \{
\begin{array}
[c]{rclcrcl}%
u_{r}(1,\theta) & = & u_{r}^{\ast}(\theta)\,, &  & \lim_{r\rightarrow \infty}{u}_{r}(r,\theta) & = & 0\,,\\[6pt]%
u_{\theta}(1,\theta) & = & u_{\theta}^{\ast}(\theta)\,, &  & \lim
_{r\rightarrow \infty}{u}_{\theta}(r,\theta) & = & 0\,,
\end{array}
\right. 
 \quad \forall \text{ }\theta \in(-\pi,\pi)\,. \label{bc_polar}%
\end{equation}
For the exact solution $(\mathbf{u}_{\mu},p_{\mu})$ given by (\ref{explicit})
we have in polar coordinates for the corresponding stream-function-vorticity
pair $(\psi_{\mu},\omega_{\mu})$, for all $\mu \in \mathbb{R}$:
\[
\left \{
\begin{array}
[c]{ccc}%
\smallskip \psi_{\mu}(r,\theta) & = & -{\mu \ln(r)}\,,\\
\omega_{\mu}(r,\theta) & = & 0\text{ },
\end{array}
\right.  \qquad \forall \text{ }(r,\theta)\in \Omega \, \text{$.$}%
\]
In order to prove Theorem \ref{thm_main} we construct, as explained above, a
solution which is a perturbation of the explicit solutions $(\mathbf{u}_{\mu
},p_{\mu})$. We therefore set $\psi=\psi_{\mu}+\gamma$ and $\omega=\omega
_{\mu}+w$. Substituting this Ansatz into (\ref{NS_polar}), (\ref{comp_polar}),
we obtain the following equivalent system for the unknowns $(\gamma,w)$:
\begin{equation}
\left \{
\begin{array}
[c]{lcl}%
\partial_{rr}\gamma+\frac{1}{r}\partial_{r}\gamma+\dfrac{1}{r^{2}}%
\partial_{\theta \theta}\gamma & = & -w\,,\\[6pt]%
\partial_{rr}w+\frac{1}{r}\partial_{r}w+\dfrac{1}{r^{2}}\partial_{\theta
\theta}w-\dfrac{\mu}{r^{2}}\partial_{\theta}w & = & {\dfrac{\partial_{\theta
}\gamma}{r}\partial_{r}w-\dfrac{\partial_{r}\gamma}{r}\partial_{\theta}%
w}\text{ },
\end{array}
\right.  \quad \text{ }\forall \text{ }(r,\theta)\in \Omega \text{ },
\label{NS_pert}%
\end{equation}
with the boundary conditions:%
\begin{equation}
\left \{
\begin{array}
[c]{lcl}%
\smallskip \partial_{\theta}\gamma(1,\theta) & = & v_{r}^{\ast}(\theta)\text{
},\\
\smallskip \partial_{r}\gamma(1,\theta) & = & -v_{\theta}^{\ast}(\theta)\text{
},\\
{\lim_{r\rightarrow \infty}}\, \left(  |\gamma(r,\theta)|+\left \vert
\partial_{r}\gamma(r,\theta)\right \vert \right)  & = & 0\text{ },
\end{array}
\right.  \quad \text{ }\forall \text{ }\theta \in(-\pi,\pi)\text{ },
\label{BC_pert}%
\end{equation}
for certain $(v_{r}^{\ast}(\theta),v_{\theta}^{\ast}(\theta))$ to be defined
later on, satisfying:%
\begin{equation}
\int_{\partial B}v_{r}^{\ast}\text{ }\mathrm{d}\sigma=0\text{ }, \label{abc}%
\end{equation}
and which are small in a sense to be made precise.

Following the method developed in \cite{HillairetWittwer09}, we solve
(\ref{NS_pert}), (\ref{BC_pert}), for data ${(v_{r}^{\ast},v_{\theta}^{\ast}%
)}$, by interpreting the radial coordinate $r$ as a time and by expanding in a
Fourier series:
\[
\gamma(r,\theta)=\sum_{n\in \mathbb{Z}}\gamma_{n}(r)e^{in\theta}\text{ },\qquad
w(r,\theta)=\sum_{n\in \mathbb{Z}}w_{n}(r)e^{in\theta}\text{ }.
\]

\paragraph{Notation.}

To unburden the notation we write for the Fourier series of $\gamma$ and $w$:%
\begin{equation}
{\hat{\gamma}=(\gamma_{n})_{n\in \mathbb{Z}}\,,\quad \hat{w}=(w_{n}%
)_{n\in \mathbb{Z}}\,,} \label{notation}%
\end{equation}
and analogously for all other functions.

\medskip

From (\ref{NS_pert}), (\ref{BC_pert}) we obtain, for $n\in \mathbb{Z}$, the
following system of ordinary differential {equations:}%
\begin{equation}
\left \{
\begin{array}
[c]{lll}%
\smallskip \partial_{rr}\gamma_{n}+\dfrac{1}{r}\partial_{r}\gamma_{n}%
-\dfrac{n^{2}}{r^{2}}\gamma_{n} & = & -w_{n}\,,\\
\partial_{rr}w_{n}+\dfrac{1}{r}\partial_{r}w_{n}-\dfrac{i\mu n+n^{2}}{r^{2}%
}w_{n} & = & F_{n}\,,
\end{array}
\right.  \quad \text{ on }(1,\infty)\,, \label{Fourier}%
\end{equation}
with the source term $F_{n}$ given by:
\begin{equation}
F_{n}= - {\dfrac{i}{r}}{\sum_{k+l=n}}\left(  k\,w_{k}\, \partial_{r}\gamma
_{l}-l\, \gamma_{l}\, \partial_{r}w_{k}\right)  \,, \label{eq_Fn}%
\end{equation}
and with the boundary conditions:%
\begin{equation}
\left \{
\begin{array}
[c]{lcl}%
\smallskip in\gamma_{n}(1) & = & v_{r,n}^{\ast}\text{ },\\
\smallskip - \partial_{r}\gamma_{n}(1) & = & v_{\theta,n}^{\ast}\text{ },\\
{\lim_{r\rightarrow \infty}}\, \left(  |\gamma_{n}(r)|+\left \vert \partial
_{r}\gamma_{n}(r)\right \vert \right)  & = & 0\text{ },
\end{array}
\right.  \quad \forall \,n\in \mathbb{Z}\setminus \{0\} \,. \label{eq_bcFn}%
\end{equation}
Note that $v_{r,0}^{\ast}=0$ by assumption (\ref{abc}) {and that} the value of
$\gamma_{0}(1)$ is irrelevant,\textit{ i.e.}, the stream function is only
unique up to an additive constant. As we {show later in this section, the value
$v_{\theta,0}^{\ast}$} cannot be chosen freely if one wants the solution
$\gamma_{0}$ to satisfy the boundary condition at infinity.

\medskip

For convenience, we {first solve \eqref{Fourier} with boundary conditions:%
\begin{equation}
\left \{
\begin{array}
[c]{lcl}%
\smallskip \gamma_{n}(1) & = & \gamma_{n}^{\ast}\text{ },\\
\smallskip w_{n}(1) & = & \omega_{n}^{\ast}\text{ },\\
{\lim_{r\rightarrow \infty}}\, \left(  |\gamma_{n}(r)|+\left \vert
w_{n}(r)\right \vert \right)  & = & 0\text{ },
\end{array}
\right.  \quad \forall \,n\in \mathbb{Z}\setminus \{0\} \,, \label{BC_Fourier}%
\end{equation}
instead of (\ref{BC_Fourier}).} Once the solution is constructed we then
re-express the solution in terms of the original boundary conditions.

Assuming that the functions $F_{n}$ are continuous and decay sufficiently
rapidly at infinity, there exits exactly one solution to (\ref{Fourier})
satisfying \eqref{BC_Fourier}. Since the Green's function of equations
(\ref{Fourier}) are $r\mapsto r^{\pm \left \vert n\right \vert }$ and $r\mapsto
r^{\pm \zeta_{n}}$, respectively, where $\zeta_{n}=\sqrt{n^{2}+i\mu n}$, with 
$\mathcal{R}e(\sqrt{z}) >0$ for $z \in \mathbb C \setminus (-\infty,0]$,  
the solutions are given by the following explicit expressions:%
\begin{equation}
\left \{
\begin{array}
[c]{lll}%
\bigskip \gamma_{n}(r) & = & \dfrac{\overline{\gamma}_{n}}{r^{|n|}}+%
%TCIMACRO{\dint _{r}^{\infty}}%
%BeginExpansion
{\displaystyle \int_{r}^{\infty}}
%EndExpansion
\dfrac{sw_{n}(s)}{2|n|}\left(  \dfrac{r}{s}\right)  ^{|n|}\text{$\,
\mathrm{d}s$}+%
%TCIMACRO{\dint _{1}^{r}}%
%BeginExpansion
{\displaystyle \int_{1}^{r}}
%EndExpansion
\dfrac{sw_{n}(s)}{2|n|}\left(  \dfrac{s}{r}\right)  ^{|n|}\text{$\,
\mathrm{d}s$}\,,\\
w_{n}(r) & = & \dfrac{\overline{w}_{n}}{r^{\zeta_{n}}}-%
%TCIMACRO{\dint _{r}^{\infty}}%
%BeginExpansion
{\displaystyle \int_{r}^{\infty}}
%EndExpansion
\dfrac{sF_{n}(s)}{2\zeta_{n}}\left(  \dfrac{r}{s}\right)  ^{\zeta_{n}%
}\text{$\, \mathrm{d}s$}-%
%TCIMACRO{\dint _{1}^{r}}%
%BeginExpansion
{\displaystyle \int_{1}^{r}}
%EndExpansion
\dfrac{sF_{n}(s)}{2\zeta_{n}}\left(  \dfrac{s}{r}\right)  ^{\zeta_{n}%
}\text{$\, \mathrm{d}s$}\,,
\end{array}
\right.  \quad \text{ for }n\in \mathbb{Z}\setminus \{0\} \,. \label{eq_wngamman}%
\end{equation}
with:%
\begin{equation}
\left \{
\begin{array}
[c]{lll}%
\medskip \overline{\gamma}_{n} & = & \gamma_{n}^{\ast}-%
%TCIMACRO{\dint _{1}^{\infty}}%
%BeginExpansion
{\displaystyle \int_{1}^{\infty}}
%EndExpansion
\dfrac{sw_{n}(s)}{2|n|}\left(  \dfrac{1}{s}\right)  ^{|n|}\text{$\,
\mathrm{d}s$}\,,\\
\overline{w}_{n} & = & w_{n}^{\ast}+%
%TCIMACRO{\dint _{1}^{\infty}}%
%BeginExpansion
{\displaystyle \int_{1}^{\infty}}
%EndExpansion
\dfrac{sF_{n}(s)}{2\zeta_{n}}\left(  \dfrac{1}{s}\right)  ^{\zeta_{n}%
}\text{$\, \mathrm{d}s\,,$}%
\end{array}
\right.  \quad \text{ for }n\in \mathbb{Z}\setminus \{0\} \,.
\label{eq_wngammanb}%
\end{equation}
For $n=0$, there still exist solutions to (\ref{Fourier}) decaying at
infinity, but these solutions exist only for exactly one boundary condition.
The reason is that for $n=0$ the Green's functions for the equations in
(\ref{Fourier}) are $r\mapsto1$ and $r\mapsto \ln r$, which do not decay at
infinity. The solutions decaying at infinity are:%
\begin{equation}
{\left \{
\begin{array}
[c]{lll}%
\medskip \gamma_{0}(r) & = &
%TCIMACRO{\dint _{r}^{\infty}}%
%BeginExpansion
{\displaystyle \int_{r}^{\infty}}
%EndExpansion
\dfrac{1}{s}%
%TCIMACRO{\dint _{s}^{\infty}}%
%BeginExpansion
{\displaystyle \int_{s}^{\infty}}
%EndExpansion
t\,w_{0}(t)\text{$\, \mathrm{d}t\, \mathrm{d}s$}\,,\\[10pt]%
w_{0}(r) & = & -%
%TCIMACRO{\dint _{r}^{\infty}}%
%BeginExpansion
{\displaystyle \int_{r}^{\infty}}
%EndExpansion
\dfrac{1}{s}%
%TCIMACRO{\dint _{s}^{\infty}}%
%BeginExpansion
{\displaystyle \int_{s}^{\infty}}
%EndExpansion
t\,F_{0}(t)\text{$\, \mathrm{d}t\, \mathrm{d}s$}\,.
\end{array}
\right.  } \label{eq_w0gamm0}%
\end{equation}
{We recall that the value of $\gamma_{0}(1)$ is irrelevant}. The value of
$w_{0}(1)$ fixes the value of {$\partial_{r}\gamma_{0}(1)= - v_{\theta,0}^{\ast}$
as a function of $\mu$. Once the solution is constructed, we will show that
$\mu \mapsto \mu - \partial_{r}\gamma_{0}(1)=:\mu^{\ast}$ can be inverted, which
then shows the existence of a solution for an open set of boundary
conditions.}

\section{Functional framework and main result}

We now introduce the function spaces which we use to solve the system
{(\ref{eq_bcFn})--(\ref{eq_w0gamm0})}. We use the notation introduced in
(\ref{notation}):

\begin{definition}
\label{spaces}Given $\kappa>0$, $\alpha>0$ and $m\in \mathbb{N}$, such that
$m<\kappa$, we set:
\[%
\begin{array}
[c]{rcl}%
\mathcal{B}_{\kappa} & := & \{ \hat{\varphi}^{\ast}\in \mathbb{C}^{\mathbb{Z}%
}\text{ such that }{\sup_{n\in \mathbb{Z}}}(1+|n|)^{\kappa}|\varphi_{n}^{\ast
}|<\infty \} \text{ },\\[10pt]%
\mathcal{B}_{\kappa}^{0} & := & \{ \hat{\varphi}^{\ast}\in \mathcal{B}_{\kappa
}\text{ such that }\varphi_{0}^{\ast}=0\} \text{ },
\end{array}
\]
and
\[%
\begin{array}
[c]{rcl}%
\mathcal{B}_{\alpha,\kappa} & := & \{ \hat{\varphi}\in({C}([1,\infty
);\mathbb{C}))^{\mathbb{Z}},\text{ such that }\ {\sup_{n\in \mathbb{Z}}\,
\sup_{r\in \lbrack1,\infty)}}r^{\alpha}(1+|n|)^{\kappa}|\varphi_{n}%
(r)|<\infty \} \text{ },\\[10pt]%
\mathcal{U}_{\alpha,\kappa}^{m} & := & \{ \hat{\varphi}\in({C}^{m}%
([1,\infty);\mathbb{C}))^{\mathbb{Z}},\text{ such that }(\partial_{r}%
^{l}\varphi_{n})_{n\in \mathbb{Z}}\in \mathcal{B}_{\alpha+l,\kappa
-l},\  \text{for all }0\leq \text{$l\leq m$}\} \text{ }.
\end{array}
\]

\end{definition}

\medskip

These function spaces are reminiscent of weighted Sobolev spaces, and permit
to obtain sharp estimates on the decay of solutions to {(\ref{eq_wngamman}%
)--(\ref{eq_w0gamm0})}. The spaces with one lower index (mainly $\mathcal{B}%
_{\kappa}^{0}$) are used for the
boundary data, whereas the spaces with two lower indices (mainly
$\mathcal{U}_{\alpha,\kappa}^{m}$) will be used for {solving
(\ref{eq_wngamman})--(\ref{eq_w0gamm0})}.\medskip

The spaces introduced in Definition \ref{spaces} satisfy the following
straightforward properties. Given $\alpha>0$, $\kappa>0$, and $m\in \mathbb{N}%
$, such that $m<\kappa$, we have:

\begin{enumerate}
\item The spaces $\mathcal{B}_{\kappa}$, $\mathcal{B}_{\alpha,\kappa}$,
$\mathcal{U}_{\alpha,\kappa}^{m}$ are Banach spaces when equipped with their
respective norms:
\[
\Vert \hat{w}\ ;\  \mathcal{B}_{\kappa}\Vert=\sup_{n\in \mathbb{N}}%
(1+|n|)^{\kappa}|w_{n}|,\qquad \Vert \hat{\varphi}\ ;\  \mathcal{B}%
_{\alpha,\kappa}\Vert=\sup_{n\in \mathbb{Z}}\sup_{r\in \lbrack1,\infty
)}r^{\alpha}(1+|n|)^{\kappa}|\varphi_{n}(r)|\,,
\]%
\[
\Vert \hat{\varphi}\ ;\  \mathcal{U}_{\alpha,\kappa}^{m}\Vert=\sum_{l=0}%
^{m}\Vert(\partial_{r}^{l}\varphi_{n})_{n\in \mathbb{Z}}\ ;\  \mathcal{B}%
_{\alpha+l,\kappa-l}\Vert \,.
\]

\item Given $\alpha \geq \alpha^{\prime}$ and $\kappa \geq \kappa^{\prime}$ we
have the embedding $\mathcal{B}_{\alpha,\kappa}\subset \mathcal{B}%
_{\alpha^{\prime},\kappa^{\prime}}$ together with the bound:
\begin{equation}
\Vert \hat{\varphi}\ ;\mathcal{B}_{\alpha^{\prime},\kappa^{\prime}}\Vert
\leq \Vert \hat{\varphi}\ ;\mathcal{B}_{\alpha,\kappa}\Vert \,,\qquad \forall \,
\hat{\varphi}\in \mathcal{B}_{\alpha,\kappa}\,. \label{eq_emb}%
\end{equation}

\item The space $\mathcal{B}_{\kappa}^{0}$ is a closed subspace of
$\mathcal{B}_{\kappa}$, and thus also a Banach space.
\end{enumerate}

We now formulate the problem of finding a solution to {(\ref{eq_bcFn}%
})--(\ref{eq_w0gamm0}) in such a way that we can apply the inverse map theorem
on our function spaces:

\begin{lemma}
\label{lem_S} Given ${\mu>\mu_{crit}}$, let $\alpha>0$ be sufficiently small and $\kappa >0$.
Then, the map $\mathcal{S}_{\mu}\colon \mathcal{B}_{4+2\alpha,\kappa}%
\times \mathcal{B}_{\kappa+4}^{0}\times \mathcal{B}_{\kappa+2}^{0}%
\rightarrow \mathcal{U}_{\alpha,\kappa+4}^{2}\times \mathcal{U}_{\alpha
+2,\kappa+2}^{2}$, which associates to the triple $(\hat{F},\hat{\gamma}%
^{\ast},\hat{w}^{\ast})$ the pair $\left(  \hat{\gamma},\hat{w}\right)  $ by
virtue of equations (\ref{eq_wngamman}), (\ref{eq_w0gamm0}), together with
(\ref{BC_Fourier}), (\ref{eq_wngammanb}), is linear and continuous.
\end{lemma}

The notion of $\alpha$ small enough will be made precise in the {last
section.}

\begin{lemma}
\label{lem_NL} Let $\alpha>0$ and $\kappa>0$. Then, the map $\mathcal{NL}%
\colon \left(  \mathcal{U}_{\alpha,\kappa+4}^{2}\times \mathcal{U}%
_{\alpha+2,\kappa+2}^{2}\right)  ^{2}\rightarrow \mathcal{B}_{4+2\alpha
,\kappa+1}$, defined by%
\[
\mathcal{NL}[(\hat{\gamma}_{a},\hat{w}_{a}),(\hat{\gamma}_{b},\hat{w}%
_{b})]= \left( r \mapsto -{\frac{i}{r}}{\sum_{k+l=n}}\left(  k\,w_{k}^{a}(r)\, \partial_{r}%
\gamma_{l}^{b}(r)-l\, \gamma_{l}^{a}(r)\, \partial_{r}w_{k}^{b}(r)\right) \right)_{n\in \mathbb Z} \,,
\]
is bilinear and continuous.
\end{lemma}

The proofs of these lemmas are postponed to Section \ref{sec_technical}. {We
also introduce the trace operator $\Gamma_{1}$:
\[%
\begin{array}
[c]{rrcl}%
{\Gamma}_{1}\colon & \mathcal{U}_{\alpha_{1},\kappa_{1}}^{2}\times
\mathcal{U}_{\alpha_{2},\kappa_{2}}^{2} & \rightarrow & \mathcal{B}%
_{\kappa_{1}-1}^{0}\times \mathcal{B}_{\kappa_{1}-1}^{0}\,,\\
& (\hat{\gamma},\hat{w}) & \longmapsto & ((in\gamma_{n}(1))_{n\in \mathbb{Z}%
},((\delta_{n,0}-1)\partial_{r}\gamma_{n}(1))_{n\in \mathbb{Z}})\,,
\end{array}
\]
where $\delta_{n,m}$ is the Kronecker symbol. This map is linear and
continuous for arbitrary $(\alpha_{i},\kappa_{i})\in(0,\infty)^{2}$,
$(i=1,2)$. } \medskip

To compute {solutions to (\ref{eq_bcFn})}--(\ref{eq_w0gamm0}), we introduce a
map $\Phi_{\mu}$, which allows to solve the differential equations and
constrain the trace on $r=1$ in one step. Namely, given {$\mu>\mu_{crit}$,} $\alpha$ sufficiently
small and $\kappa>0$, we set:
\begin{equation}%
\begin{array}
[c]{rrcl}%
\Phi_{\mu}\colon & (\mathcal{U}_{\alpha,\kappa+4}^{2}\times \mathcal{U}%
_{\alpha+2,\kappa+2}^{2})\times(\mathcal{B}_{\kappa+4}^{0}\times
\mathcal{B}_{\kappa+2}^{0}) & \longrightarrow & (\mathcal{U}_{\alpha,\kappa
+4}^{2}\times \mathcal{U}_{\alpha+2,\kappa+2}^{2})\times(\mathcal{B}_{\kappa
+3}^{0}\times \mathcal{B}_{\kappa+3}^{0})\, \\[10pt]
&
\begin{pmatrix}
\hat{x}=(\hat{\gamma},\hat{w})\\[4pt]%
\hat{x}^{\ast}=(\hat{\gamma}^{\ast},\hat{w}^{\ast})
\end{pmatrix}
& \longmapsto &
\begin{pmatrix}
\mathcal{S}_{\mu}(\mathcal{NL}(\hat{x},\hat{x}),\hat{x}^{\ast})-\hat{x}\\[4pt]%
\Gamma_{1}[\mathcal{S}_{\mu}(\mathcal{NL}(\hat{x},\hat{x}),\hat{x}^{\ast})]
\end{pmatrix}
\end{array}
\label{PhiMu}%
\end{equation}
By definition, if $(\hat{x}=(\hat{\gamma},\hat{w}),\hat{x}^{\ast}=(\hat
{\gamma}^{\ast},\hat{w}^{\ast}))$ is a solution to $\Phi_{\mu}(\hat{x},\hat
{x}^{\ast})=(0,(\hat{v}_{r}^{\ast},\hat{v}_{\theta}^{\ast}))$, then
$(\hat{\gamma},\hat{w})$ satisfies
{(\ref{Fourier})-\eqref{eq_bcFn}}. This motivates the following notion of
$\kappa$--solutions:

\begin{definition}
Given an exponent $\kappa>0$, an angular velocity {$\mu>\mu_{crit}$}, and a
boundary condition $\hat{\mathbf{v}}^{\ast}:=(\hat{v}_{r}^{\ast},\hat
{v}_{\theta}^{\ast})\in \mathcal{B}_{\kappa+3}^{0}\times \mathcal{B}_{\kappa
+3}^{0}$, we call {\emph{$\kappa$--solution for the boundary condition
$\hat{\mathbf{v}}^{\ast}$ and the asymptotic angular velocity $\mu$}}\ a pair
$\hat{x}=(\hat{\gamma},\hat{w})$, such that, for sufficiently small $\alpha>0$
and some $\hat{x}^{\ast}=(\hat{\gamma}^{\ast},\hat{w}^{\ast})\in
\mathcal{B}_{\kappa+4}^{0}\times \mathcal{B}_{\kappa+2}^{0}$:

\begin{itemize}
\item $\hat{x}\in \mathcal{U}_{\alpha+4,\kappa}^{2}\times \mathcal{U}%
_{\alpha+2,\kappa+2}^{2},$

\item $\Phi_{\mu}(\hat{x},\hat{x}^{*}) = (0,\hat{\mathbf{v}}^{*}).$
\end{itemize}
\end{definition}

The remaining sections are devoted to the proof of the following result:

\begin{theorem}
\label{thm_main_tech} Given $\kappa>0$ and {$\mu_{0}>\mu_{crit}$} there exists
$\varepsilon_{\kappa,\mu_{0}}>0$ and an open interval $I_{\kappa,\mu_{0}}%
\ni \mu_{0}$ such that, given $\hat{\mathbf{v}}^{\ast}\in B(\mathcal{B}%
_{\kappa+3}^{0}\times \mathcal{B}_{\kappa+3}^{0}\ ;\  \varepsilon_{\kappa
,\mu_{0}})$ and $\mu^{\ast}\in$ ${I_{\kappa,\mu_{0}}}$, there exists {$\mu
>\mu_{crit}$} and a $\kappa$--solution $(\hat{\gamma},\hat{w})$ for the
boundary condition $\hat{\mathbf{v}}^{\ast}$ and the asymptotic angular
velocity $\mu$, satisfying the condition $\mu - \partial_{r}\gamma_{0}(1)
=\mu^{\ast}$.
\end{theorem}

As mentioned above, the notion of {$\alpha$} small enough will be made precise
in the last section. Before entering into the details of the proof of 
Theorem \ref{thm_main_tech}, we explain why it implies Theorem \ref{thm_main}.

\medskip

\begin{proof}
[Proof of {Theorem \ref{thm_main}}]Let {$\mu_{0}>\mu_{crit}$} and {$\kappa>1$%
}. Applying Theorem \ref{thm_main_tech} yields a ball of initial conditions
with positive radius $\varepsilon_{\kappa,\mu_{0}}>0$ and an open neighborhood
$I_{\kappa,\mu_{0}}$ of $\mu_{0}$.

\medskip

For $\mathbf{u}^{\ast}\in C^{\infty}(\partial B),$ we define : 
$$
\mu^{\ast} = \dfrac{1}{2\pi}\int_{0}^{2\pi}u_{\theta}^{\ast}(s)\,
\mathrm{d}\text{$s$}\,,
$$
and the sequences $\hat{u}^{\ast}$ and $\hat{v}^{\ast}$ by $u_0=v_0 =0$ and:
$$
u_{n}^{\ast}  =   \dfrac{1}{2\pi}\int_{0}^{2\pi}u_{r}^{\ast}(\theta)e^{-in\theta
}\text{ \textrm{d}$\theta$}\,,\qquad v_{n}^{\ast}= \dfrac{1}{2\pi}\int
_{0}^{2\pi}u_{\theta}(\theta)e^{-in\theta}\text{ \textrm{d}$\theta$}%
\,,\qquad \forall \,n\in \mathbb{Z}\setminus \{0\} \,.
$$
The
regularity of $\mathbf{u}^{\ast}$ yields that $(\hat{u}^{\ast},\hat{v}^{\ast
})\in \mathcal{B}_{\kappa+3}^{0}\times \mathcal{B}_{\kappa+3}^{0}$ and 
\[
u_{r}^{\ast}(\theta)= \sum_{n\in \mathbb{Z}}u_{n}^{\ast}e^{in\theta}\,,\qquad
u_{\theta}^{\ast}(\theta)=\mu^{\ast} + \sum_{n\in \mathbb{Z}}v_{n}^{\ast
}e^{in\theta}\,.
\] 
We now assume that:
\begin{equation}
\Vert(\hat{u}^{\ast},\hat{v}^{\ast})\ ;\  \mathcal{B}_{\kappa+3}^{0}%
\times \mathcal{B}_{\kappa+3}^{0}\Vert<\varepsilon_{\kappa,\mu_{0}}\,,\quad
\mu^{\ast}\in I_{\kappa,\mu_{0}}\,, \label{eq_2}%
\end{equation}
which makes the meaning of $\mathbf{u}^{\ast}$ sufficiently close to $\mu
_{0}\mathbf{e}_{\theta}$ in the statement of Theorem \ref{thm_main} precise.

\medskip

Consequently, the assumptions of Theorem \ref{thm_main_tech} are satisfied,
which yields that there exists {$\mu>\mu_{crit}$} and a $\kappa$--solution $(\hat{\gamma}%
,\hat{w})$ for the boundary data {$(\hat{u}^{\ast},\hat{v}^{\ast})$}
satisfying $\mu - \partial_{r}\gamma_{0}(1)=\mu^{\ast}$. Let
\[
w(r,\theta)=\sum_{n\in \mathbb{Z}}w_{n}(r)e^{in\theta}\,,\qquad \gamma
(r,\theta)=\sum_{n\in \mathbb{Z}}\gamma_{n}(r)e^{in\theta}\,,\quad
\forall{\,(r,\theta)\in \Omega \,.}%
\]
Because $\kappa>1$, classical results from the theory of Fourier series yield that:

\begin{itemize}
\item $w\in C^{2}(\mathbb{R}^{2}\setminus \overline{B})$ and
$\gamma \in C^{4}(\mathbb{R}^{2}\setminus \overline{B})$ so that
$\mathbf{v}=\nabla^{\bot}\gamma \in C^{3}(\mathbb{R}^{2}%
\setminus \overline{B})$,

\item $\mathbf{v}\cdot \nabla w\in C^{1}(\mathbb{R}^{2}\setminus
\overline{B})$ with:%
\[
\mathbf{v}\cdot \nabla w(r,\theta)=\sum_{n\in \mathbb{Z}}\left[  - \frac{i}{r}{\sum_{k+l=n}%
}\left(  l\,w_{l}(r)\, \partial_{r}\gamma_{k}(r)-k\, \gamma_{k}(r)\, \partial_{r}%
w_{l}(r)\right)  \right]  e^{in\theta}\,, \quad \forall \, (r,\theta) \in \Omega\,.
\]
\end{itemize}

Let
\[
\Delta{w}-{\dfrac{\mu}{r}}\  \mathbf{e}_{\theta}\cdot \nabla w-\mathbf{v}%
\cdot \nabla w=\colon \varphi \in{C}([1,\infty);{C}([-\pi,\pi]))\,.
\]
Because $(\hat{\gamma},\hat{w})$ satisfies (\ref{Fourier}), all the Fourier
coefficients of $\varphi$ vanish identically on $[1,\infty)$. Hence,
$(\gamma,w)$ is a solution of
\begin{equation}
\left \{
\begin{array}
[c]{rcl}%
\Delta \gamma & = & - w\,,\\
\Delta{w}-{\dfrac{\mu}{r}}\  \mathbf{e}_{\theta}\cdot \nabla w & = &
\mathbf{v}\cdot \nabla w\,,
\end{array}
\right.  \quad \text{in $\mathbb{R}^{2}\setminus \overline{B}\,.$} \label{eq_3}%
\end{equation}
Straightforward manipulations of the Fourier series of $\gamma$ yield that:
\begin{equation}
|\mathbf{v}(r,\theta)|\leq \dfrac{\Vert \gamma \ ;\  \mathcal{U}_{\alpha,\kappa
+4}^{2}\Vert}{r^{\alpha+1}}\,,\quad \forall \,{(r,\theta)\in \Omega}\,.
\label{eq_decayv}%
\end{equation}
Since $\mathbf{u}^{\ast}$ has zero flux,\textit{ i.e.}, since
$v^*_{r}$ has zero average, we have that $\mathbf{v}(1,\theta)=v_{r}^{\ast
}(\theta)e_{r}+v_{\theta}^{\ast}(\theta)\mathbf{e}_{\theta},$ for all
$\theta \in (-\pi,\pi)$, where:
\begin{align*}
v_{r}^{\ast}(\theta)  &  = \sum_{n\in \mathbb{Z}}in\gamma_{n}(1)e^{in\theta
}= \sum_{n\in \mathbb{Z}}u_{n}^{\ast}e^{in\theta}=u_{r}^{\ast}(\theta)\,,\\
v_{\theta}^{\ast}(1,\theta)  &  = - \sum_{n\in \mathbb{Z}}\partial_{r}\gamma
_{n}(1)e^{in\theta}=u_{\theta}^{\ast}(\theta) - \partial_{r}\gamma_{0}%
(1)-\mu^{\ast}\,.
\end{align*}
Therefore, if we set $\psi:=\psi_{\mu}+\gamma$, $\omega:=w$, $\mathbf{u}%
:=\nabla^{\bot}\psi=\mathbf{u}_{\mu}+\mathbf{v}$, then the pair $(\psi
,\omega)$ is a solution to
\begin{equation}
\left \{
\begin{array}
[c]{rcl}%
\Delta \psi & = & -\omega \,,\\
\Delta \omega & = & \mathbf{u}\cdot \nabla \omega \,,
\end{array}
\right.  \quad \text{in $\mathbb{R}^{2}\setminus \overline{B}$}\,, \label{eq_4}%
\end{equation}
and the following boundary conditions are satisfied (recall that $\mu - \partial_r \gamma_0(1)=\mu^{\ast}$ by construction of $\hat{\gamma}$):
\[
\mathbf{u}=\mathbf{u}^{\ast}\,,\quad \text{on $\partial B\,,$}\qquad
\lim_{r\rightarrow \infty}|\mathbf{u}(r,\theta)|=0\,.
\]
The inequality \eqref{eq_decayv} implies that the boundary
condition at infinity is satisfied in the following more precise sense:
\[
\lim_{r\rightarrow \infty}\left \Vert r\left(  \mathbf{u}(r,\theta)-\dfrac
{\mu \mathbf{e}_{\theta}}{r}\right)  ;{L^{\infty}((-\pi,\pi))}\right \Vert
=\lim_{r\rightarrow \infty}r\Vert v(r,\theta);{L^{\infty}((-\pi,\pi))}%
\Vert=0\,.
\]

\medskip

To complete the proof, we need to show how to obtain the Navier Stokes
equations \eqref{NS} from the relations between $\mathbf{u}$, $\psi$ and
$\omega$, together with \eqref{eq_4}. First, multiplying (\ref{eq_4}) by
$\varphi \in C_{c}^{\infty}(\mathbb{R}^{2}\setminus \overline{B})$
yields:
\[
{-\int_{\mathbb{R}^{2}\setminus \overline{B}}\nabla w\colon \nabla \varphi}%
=\int_{\mathbb{R}^{2}\setminus \overline{B}}\left[  \mathbf{u}\cdot \nabla
w\right]  \varphi \,.
\]
We have the basic identities:
\begin{equation}
w=\nabla \times\mathbf{u}\,,\qquad \mathbf{u}\cdot \nabla
w=\nabla \times\left[  \mathbf{u}\cdot \nabla \mathbf{u}\right]  \,,
\label{eq_identities}%
\end{equation}
from which we obtain, after integration by parts, that for any given
$\varphi \in C_{c}^{\infty}(\mathbb{R}^{2}\setminus \overline{B})$:
\begin{equation}
{-\int_{\mathbb{R}^{2}\setminus \overline{B}}\nabla \mathbf{u}\colon \nabla
\nabla^{\bot}\varphi}=\int_{\mathbb{R}^{2}\setminus \overline{B}}\left[
\mathbf{u}\cdot \nabla \mathbf{u}\right]  \cdot \nabla^{\bot}\varphi \,.
\label{eq_0}%
\end{equation}
This identity yields the pressure $p$ \textit{via} De Rham's theory, modulo
the difficulty, that not all the divergence-free velocity-fields
$\mathbf{w}\in C_{c}^{\infty}(\mathbb{R}^{2}\setminus
\overline{B})$ of compact support can be written in the form
$\nabla^{\bot}\varphi$ with $\varphi \in C_{c}^{\infty}(\mathbb{R}%
^{2}\setminus \overline{B})$. More precisely, if a smooth velocity-field
$\mathbf{v}$ satisfies $\nabla \times\mathbf{v}=0$ in $\mathbb{R}%
^{2}\setminus \overline{B}$, then $\mathbf{v}$ is the gradient of a 
function up to a contribution of the form $C\mathbf{x}^{\bot
}/|\mathbf{x}|^{2}$. We now show that this contribution vanishes in our case.

\medskip

Let $\Phi_{0}\in C^{\infty}(\mathbb{R})$ be such that $\mathrm{supp}%
(\Phi_{0}^{\prime})\subset \subset(1,2)$ and $\Phi_0(0) = 0,$ $\Phi_0(2)=1,$
 and let $\mathbf{w}_{0}%
=\nabla^{\bot}\Phi_{0}$. In polar coordinates we have $\mathbf{w}_{0}%
(r,\theta)=-\Phi_{0}^{\prime}(r)\mathbf{e}_{\theta}$, for all $(r,\theta
)\in \Omega$. Given a divergence-free $\mathbf{w}\in 
C_{c}^{\infty}(\mathbb{R}^{2}\setminus \overline{B})$ we
define $\mathbf{\widetilde{w}}$ and $\varphi$ by:
\[
\mathbf{\widetilde{w}}=\mathbf{w}-M_{\mathbf{w}}\mathbf{w}_{0}\quad \text{with}%
\quad \left[  \int_{1}^{\infty}w_{\theta}(r,0)\, \mathrm{d}r\right]  =\colon
M_{\mathbf{w}},\qquad \varphi(r,\theta)=\int_{0}^{r}{w}_{\theta}(s,\theta
)\text{$\, \mathrm{d}s$}-M_{\mathbf{w}}\Phi_{0}(r)\,.
\]
By definition of $M_{\mathbf{w}}$, we have that $\varphi \in C%
_{c}^{\infty}(\mathbb{R}^{2}\setminus \overline{B})$ and $\nabla^{\bot}%
\varphi=\mathbf{\widetilde{w}}$, so that \eqref{eq_0} implies:
\[
{-\int_{\mathbb{R}^{2}\setminus \overline{B}}\nabla \mathbf{u}\colon \nabla
\widetilde{\mathbf{w}}}=\int_{\mathbb{R}^{2}\setminus \overline{B}}\left[
\mathbf{u}\cdot \nabla \mathbf{u}\right]  \cdot \widetilde{\mathbf{w}}\,.
\]
Replacing $\widetilde{\mathbf{w}}$ by its definition yields that, for any
divergence-free $\mathbf{w}\in  C_{c}^{\infty}(\mathbb{R}%
^{2}\setminus \overline{B})$, we have:
\[
{-\int_{\mathbb{R}^{2}\setminus \overline{B}}\nabla \mathbf{u}\colon
\nabla{\mathbf{w}}}=\int_{\mathbb{R}^{2}\setminus \overline{B}}\left[
\mathbf{u}\cdot \nabla \mathbf{u}\right]  \cdot{\mathbf{w}-}M_{\mathbf{w}}%
\int_{\mathbb{R}^{2}\setminus \overline{B}}\left(  \left[  \mathbf{u}%
\cdot \nabla \mathbf{u}\right]  \cdot{\mathbf{w}_{0}+\nabla \mathbf{u}%
\colon \nabla{\mathbf{w}_{0}}}\right)  \,.
\]
Let  $I_{0}$ be the last integral in the previous equality. We then have:
\[
I_{0}=\lim_{N\rightarrow \infty}\int_{B(\mathbb{R}^{2},N)\setminus \overline{B}%
}\left(  \left[  \mathbf{u}\cdot \nabla \mathbf{u}\right]  \cdot{\mathbf{w}%
_{0}+\nabla \mathbf{u}\colon \nabla{\mathbf{w}_{0}}}\right)  \,.
\]
Integrating by parts, we obtain, for all $N>1$:
\begin{align*}
\int_{B(\mathbb{R}^{2},N)\setminus \overline{B}}\left(  \left[  \mathbf{u}%
\cdot \nabla \mathbf{u}\right]  \cdot{\mathbf{w}_{0}+\nabla \mathbf{u}%
\colon \nabla{\mathbf{w}_{0}}}\right)   &  =\int_{\partial B(\mathbb{R}^{2}%
,N)}\left(  \left[  \Phi_{0}\mathbf{u}\cdot \nabla \mathbf{u}\right]
\cdot \mathbf{n}^{\bot}{+{\Phi_{0}^{\prime}}\partial_{r}\mathbf{u}%
\cdot \mathbf{n}^{\bot}}\right)  \text{ \textrm{d}$\sigma$}\\
&  {-\int_{B(\mathbb{R}^{2},N)}\left(  \Phi_{0}\mathbf{u}\cdot \nabla
\omega+\nabla \omega \cdot \nabla \Phi_{0}\right)  }\\
&  =\int_{\partial B(\mathbb{R}^{2},N)}\left(  \Phi_{0}\left[  {(\mathbf{u}%
\cdot \nabla \mathbf{u})\cdot \mathbf{n}^{\bot}-\partial_{\mathbf n}\omega}\right]
{+{\Phi_{0}^{\prime}}}\partial_{r}\mathbf{u}\cdot \mathbf{n}^{\bot}\right)
\text{ \textrm{d}$\sigma$}\,,
\end{align*}
where, in order to get the last identity, we have again used that
$\mathbf{u}\cdot \nabla \omega=\Delta \omega$, in $\mathbb{R}^{2}\setminus
\overline{B}$. Since $\mathbf{u}$ decays like $1/r$, $\nabla \mathbf{u}$ decays
like $1/r^{2}$, and $\nabla \omega$ like $1/r^{3}$. This yields that $I_{0}=0$ in
the limit $N\rightarrow \infty$. Finally, we have:
\[
{-\int_{\mathbb{R}^{2}\setminus \overline{B}}\nabla \mathbf{u}\colon
\nabla{\mathbf{w}}}=\int_{\mathbb{R}^{2}\setminus \overline{B}}\left[
\mathbf{u}\cdot \nabla \mathbf{u}\right]  \cdot{\mathbf{w}}\,,
\]
for any divergence-free vector-field $\mathbf{w}\in  C%
_{c}^{\infty}(\mathbb{R}^{2}\setminus \overline{B})$, and De
Rham's theory (see \cite[Remark 1.5]{TemamB}) implies the existence of a
pressure $p$ such that {\eqref{NS} }is satisfied.
\end{proof}

\section{Proof of \textbf{Theorem \ref{thm_main_tech}}}

In this section $\kappa>0$ and {$\mu_{0}>\mu_{crit}$} are fixed. First, we set
$\mu_{-}=\left(  \mu_{0}+{\mu_{crit}}\right)  /2$ and $\mu_{+}=(2\mu_{0}+$
${\mu_{crit}})/2$ so that $I:=[\mu_{-},\mu_{+}]$ satisfies%
\[
\mu_{0}\in \lbrack \mu_{-},\mu_{+}]\subset \left(  {\mu_{crit}},\infty \right)
\,.
\]
We also set:
\begin{equation} \label{eq_alpha}
\alpha:=\dfrac{1}{4}\min \left(  \frac{1}{\sqrt{2}}\left[  \sqrt{ 1+{|\mu_{-}|%
}^{2}  }+1\right]  ^{{1}/{2}}-2,1\right)  \,.
\end{equation}
We emphasize that, because ${\mu_{0}>\mu_{crit}}$, we have $\alpha\in
(0,1/4]$. Let:
\begin{equation}%
\begin{array}
[c]{rrcl}%
\Phi_{\mu}\colon & (\mathcal{U}_{\alpha,\kappa+4}^{2}\times \mathcal{U}%
_{\alpha+2,\kappa+2}^{2})\times(\mathcal{B}_{\kappa+4}^{0}\times
\mathcal{B}_{\kappa+2}^{0}) & \longrightarrow & (\mathcal{U}_{\alpha
,\kappa+4}^{2}\times \mathcal{U}_{\alpha+2,\kappa+2}^{2})\times
(\mathcal{B}_{\kappa+3}^{0}\times \mathcal{B}_{\kappa+3}^{0})\, \\[10pt]
&
\begin{pmatrix}
\hat{x}=(\hat{\gamma},\hat{w})\\[4pt]%
\hat{x}^{\ast}=(\hat{\gamma}^{\ast},\hat{w}^{\ast})
\end{pmatrix}
& \longmapsto &
\begin{pmatrix}
\mathcal{S}_{\mu}(\mathcal{NL}(\hat{x},\hat{x}),\hat{x}^{\ast})-\hat{x}\\[4pt]%
\Gamma_{1}[\mathcal{S}_{\mu}(\mathcal{NL}(\hat{x},\hat{x}),\hat{x}^{\ast})]
\end{pmatrix}
\end{array}
\label{mapPHI}%
\end{equation}
We will show in the next section that for $\mu \in I$ the map $\Phi_{\mu}$ is
well defined. We split the proof of Theorem~\ref{thm_main_tech} into two
steps. First, we show that we can construct a $\kappa$--solution for any
sufficiently small boundary condition $(\hat{u}^{\ast},\hat{v}^{\ast}%
)\in \mathcal{B}_{\kappa+3}^{0}\times \mathcal{B}_{\kappa+3}^{0}$ and an
interval of asymptotic  angular velocities $\mu$. In a second step, we analyze the
dependence of this solution as a function of $\mu$.

\medskip

We have the following abstract result:

\begin{proposition}
\label{prop_abstractresult} Let $I$ be a compact interval and $X$, $Y$ two
Banach spaces. Assume that $\Phi \colon \mu \in I\mapsto \Phi_{\mu}$ satisfies:

\begin{itemize}
\item $\Phi \in{C}( I ; {C}^{1}(X ; Y))$,

\item $\Phi_{\mu}(x) = 0$ for all $\mu \in I$,

\item $D\Phi_{\mu}(0)$ is one-to-one and onto, and has a continuous inverse for all
$\mu \in I.$
\end{itemize}

\noindent Then, there exist positive constants $\eta_{x}$ and $\eta_{y}$, such
that $\Phi_{\mu}$ is a $C^{1}$-diffeomorphism from $B_{X}(0,\eta_{x})$ onto
$\Phi_{\mu}\left(  B_{X}(0,\eta_{x})\right)  \supset B_{Y}(0,\eta_{y})$.
Furthermore, the family of inverse maps $\Phi^{-1}\colon \mu \in I\mapsto \left(
\Phi_{\mu}\right)  ^{-1}$ satisfies
\[
\Phi^{-1}\in C(I;C^{1}(B_{Y}(0,\eta_{y});B_{X}(0,\eta_{x})))\,.
\]

\end{proposition}

\begin{proof}
The proof is standard, but for the sake of completeness we recall the main
ingredients. Given the hypothesis of Proposition \ref{prop_abstractresult},
the map $\Phi_{\mu}$ satisfies the assumptions of
the inverse function theorem for arbitrary $\mu \in I$, so that there exits $\eta_{x,\mu}>0$ and
$\eta_{y,\mu}>0$ such that $\Phi_{\mu}$ is a $C^{1}$-diffeomorphism from
$B_{X}(0,\eta_{x,\mu})$ onto $\Phi_{\mu}\left(  B_{X}(0,\eta_{x,\mu})\right)
\supset B_{Y}(0,\eta_{\mu,y})$. Since $\Phi$ is continuous, it is clear that
these constants can be chosen independently of $\mu$, locally in $\mu$. By a
compactness argument, we can therefore find constants $\eta_{x}>0$ and
$\eta_{y}>0$ such that $\Phi_{\mu}\colon B_{X}(0,\eta_{x})\rightarrow
\Phi_{\mu}\left(  B_{X}(0,\eta_{x})\right)  \supset B_{Y}(0,\eta_{y})$ is a
$C^{1}$ diffeomorphism for arbitrary $\mu \in I$. We now show that $\Phi
^{-1}\in C(I\times B_{Y}(0,\eta_{y}))$.  The proof that $\Phi^{-1}\in
C(I ; C^{1}(B_{Y}(0,\eta_{y});B_{X}(0,\eta_{x})))$ is then obtained by
differentiating  (with respect to $x$) the identity $\Phi^{-1}_{\mu} \circ \Phi_{\mu}(x) = x$ which
holds true on $B_{X}(0,\eta_{x}).$

Given $(\mu,\tilde{\mu})\in I^{2}$ and $(y,\tilde{y})\in \lbrack B_{Y}%
(0,\eta_{y})]^{2}$, we denote:
\[
x=\Phi_{\mu}^{-1}({y})\,,\qquad \tilde{x}=\Phi_{\tilde{\mu}}^{-1}({\tilde{y}})\,.
\]
By construction we have:
\begin{align*}
y-\tilde{y}  &  =\Phi_{\mu}(x)-\Phi_{\tilde{\mu}}(\tilde{x}),\\
&  =D\Phi_{\mu}(0)[x-\tilde{x}]+\Phi_{\mu}(\tilde{x})-\Phi_{\tilde{\mu}%
}(\tilde{x})+o(|x-\tilde{x}|)\\
&  =D\Phi_{\mu}(0)[x-\tilde{x}]+o(|\mu-\tilde{\mu}|)+o(\Vert x-\tilde
{x};X\Vert)\,.
\end{align*}
Consequently, reducing the size of $\eta_{x}$ and $\eta_{y}$ if necessary, we
get that:
\[
\Vert x-\tilde{x} \, ; \, X \Vert \leq2\Vert \lbrack D\Phi_{\mu}(0)]^{-1};\mathcal{L}%
_{c}(Y;X)\Vert \left[  \Vert y-\tilde{y};Y\Vert+o(|\mu-\tilde{\mu}|)\right]
\,,
\]
where $\mathcal{L}_c(Y;X)$ denotes the set of continuous linear map $Y \to X.$
This completes the proof.
\end{proof}

\medskip

We now show that we can apply Proposition \ref{prop_abstractresult} to the map
$\Phi$ as defined in
(\ref{mapPHI}) To this end, we remark that $\Phi_{\mu}$ depends on $\mu$ only
through $\mu \mapsto \mathcal{S}_{\mu}$, and that for all $\mu \in I$ the
differential $D\Phi_{\mu}(0)$ is:
\[
D\Phi_{\mu}(0)[\hat{x},\hat{x}^{\ast}]=(\Gamma_{1}\mathcal{S}_{\mu}(0,\hat
{x}^{\ast}),-\hat{x})\,,\quad \forall \,(\hat{x},\hat{x}^*)\in(\mathcal{U}%
_{\alpha,\kappa+4}^{2}\times \mathcal{U}_{\alpha+2,\kappa+2}^{2}%
)\times(\mathcal{B}_{\kappa+4}^{0}\times \mathcal{B}_{\kappa+2}^{0})\,.
\]
Let $\mathcal{T}_{\mu}(\hat{x}^*):=\Gamma_{1}%
\mathcal{S}_{\mu}(0,\hat{x}^*)$. Since $\Phi$ is a combination of linear and
bilinear maps, it suffices to apply Lemma \ref{lem_NL} and the
following lemma in order to check that the assumptions of Proposition
\ref{prop_abstractresult} are satisfied:

\begin{lemma}
\label{lem_Smu} Let $\alpha$ be given by \eqref{eq_alpha}, then the restriction of the map $\mathcal{S}\colon \mu
\mapsto \mathcal{S}_{\mu}$ to $I:=[\mu_{-},\mu_{+}]$ satisfies:

\begin{itemize}
\item[\textbf{i)}] $\mathcal{S}\in C(I;\mathcal{L}_{c}(\mathcal{B}%
_{\kappa,2\alpha+4}\times(\mathcal{B}_{\kappa+4}^{0}\times \mathcal{B}%
_{\kappa+2}^{0})\ ;\  \mathcal{U}_{\alpha,\kappa+4}^{2}\times
\mathcal{U}_{\alpha+2,\kappa+2}^{2}))$,

\item[\textbf{ii)}] $\mathcal{T}_{\mu}$ is a one to one and onto map $\mathcal{B}%
_{\kappa+4}^{0}\times \mathcal{B}_{\kappa+2}^{0}\rightarrow \mathcal{B}%
_{\kappa+3}^{0}\times \mathcal{B}_{\kappa+3}^{0}$ with continuous inverse.
\end{itemize}
\end{lemma}

We postpone the proof of this technical lemma to the next section. We now
apply Proposition \ref{prop_abstractresult} to $\Phi$, but restrict the image
to the component $\hat{x}$. This yields the following result:

\begin{theorem}
There exists a map $\Psi \colon \mu \mapsto \Psi_{\mu}$ satisfying
\[
\Psi \in C(I;C^{1}(B_{\mathcal{B}_{\kappa+3}^{0}\times \mathcal{B}_{\kappa
+3}^{0}}(0,\varepsilon_{\kappa})\ ;\  \mathcal{U}_{\alpha,\kappa+4}%
^{2}\times \mathcal{U}_{\alpha+2,\kappa+2}^{2}))\, \,,\newline%
\]
such that, for all $\mathbf{v}^{\ast}\in B_{\mathcal{B}_{\kappa+3}^{0}%
\times \mathcal{B}_{\kappa+3}^{0}}(0,\varepsilon_{\kappa})$ and $\mu \in I$,
$\Psi_{\mu}(\mathbf{v}^{\ast})$ is a $\kappa$--solution for the boundary
condition $\mathbf{v}^{\ast}$ with respect to the asymptotic angular velocity
$\mu$.
\end{theorem}

In a final step we {show how} to prescribe the zero mode of the solution
$\Psi_{\mu}(\mathbf{v}^{\ast})$ by using the dependence of the solution on
$\mu$. Let $\eta=(\mu_{0}-$ ${\mu_{crit}})/4,$ and consider a boundary
condition
\[
\mathbf{v}^{\ast}\in B_{\mathcal{B}_{\kappa+3}^{0}\times \mathcal{B}_{\kappa
+3}^{0}}(0,\varepsilon_{\kappa})\,.
\]
{Let $(\hat{\psi}_{\mu},\hat{w}_{\mu}):=\Psi_{\mu}(\mathbf{v}^{\ast}),$ for
all $\mu \in I$}. Using that $\Psi_{\mu}(0)=0$, and restricting the size of
$\varepsilon_{\kappa}$ if necessary, we can assume that
\[
|\partial_{r}\psi_{\mu,0}(1)|\leq \Vert \hat{\psi}_{\mu}\,;\, \mathcal{U}%
_{\alpha,\kappa}^{2}\Vert \leq \eta \,,\qquad \forall \, \mu \in I\,.
\]
{Consequently}, the map $\mu \mapsto \mu-\partial_{r}\psi_{\mu,0}(0)$, {which is
continuous from $I$ to $\mathbb{R}$, because $\mu \mapsto \Psi_{\mu
}(0,\mathbf{v}^{\ast})$ is continuous,} satisfies:
\[
\mu_{-}-\partial_{r}\psi_{\mu_{-},0}(1)\leq \mu_{-}+\eta<\mu_{0}\,,\qquad
\mu_{+}-\partial_{r}\psi_{\mu_{+},0}(0)\geq \mu_{+}-\eta>\mu_{0}\,.
\]
Hence the image of this map contains an open interval $I_{\kappa,\mu_{0}}$
containing $\mu_{0}$. This completes the proof of Theorem \ref{thm_main_tech}.

\section{Proof of main lemmas}

\label{sec_technical}This section contains the proof of the technical lemmas
which have been used without proof in the previous sections. First we prove
Lemma \ref{lem_NL}, which is standard. We then give  proofs of Lemmas
\ref{lem_S} and \ref{lem_Smu} which are more delicate.

\medskip

\begin{proof}
[Proof of {Lemma \ref{lem_NL}}]Let  $\hat{F} = \mathcal{NL}%
((\hat{\gamma}^{a},\hat{w}^{a}),(\hat{\gamma}^{b},\hat{w}^{b}))$, for
$(\hat{\gamma}^{i},\hat{w}^{i})\in \mathcal{U}_{\alpha,\kappa+4}^{2}%
\times \mathcal{U}_{\alpha+2,\kappa+2}^{2}$, $i=\{a,b\}$. First, we note that
for $\kappa>0$ and $n\in \mathbb{N}$, the following series converge:
\begin{equation}
\sum_{l\in \mathbb{Z}}\dfrac{1}{(1+|l|)^{\kappa+1}(1+|n-l|)^{\kappa+3}%
}\,,\qquad \sum_{k\in \mathbb{Z}}\dfrac{1}{(1+|k|)^{\kappa+3}(1+|n-k|)^{\kappa
+1}}\,. \label{series}%
\end{equation}
Consequently, we have for all $(n,k,l)\in \mathbb{Z}^{3}$:%
\[
|lw_{l}^{a}(r)\partial_{r}\gamma_{n-l}^{b}(r)|\leq \dfrac{1}{r^{3+2\alpha}}%
\dfrac{\Vert(w_{n}^{a})_{n\in \mathbb{Z}}\ ;\  \mathcal{B}_{\alpha+2,\kappa
+2}\Vert \, \Vert(\partial_{r}\gamma_{n}^{b})_{n\in \mathbb{Z}}\ ;\  \mathcal{B}%
_{\alpha+1,\kappa+3}\Vert}{(1+|l|)^{\kappa+1}(1+|n-l|)^{\kappa+3}}\,,
\]
and
\[
|k\gamma_{k}^{a}(r)\partial_{r}w_{n-k}^{b}(r)|\leq \dfrac{1}{r^{3+2\alpha}}%
\dfrac{\Vert(\gamma_{n}^{a})_{n\in \mathbb{Z}}\ ;\  \mathcal{B}_{\alpha
,\kappa+4}\Vert \, \Vert(\partial_{r}w_{n}^{b})_{n\in \mathbb{Z}}%
\ ;\  \mathcal{B}_{\alpha+3,\kappa+1}\Vert \ }{(1+|k|)^{\kappa+3}%
(1+|n-k|)^{\kappa+1}}\,,
\]
and therefore the series defining $F_{n}$ is converging. We now bound the
series (\ref{series}). By symmetry, it is sufficient to consider only the
first series and $n\geq0$. We split the sum into
two parts:%
\begin{align*}
S_{n/2}^{-}  &  ={\sum_{l=-\infty}^{\left[  n/2\right]  }}\dfrac
{1}{(1+|l|)^{\kappa+1}(1+|n-l|)^{\kappa+3}}\\
&  \leq C_{\kappa}\left(  \dfrac{1}{1+|n|}\right)  ^{\kappa+3}{\sum
_{l\in \mathbb{Z}}}\dfrac{1}{(1+|l|)^{\kappa+1}}  \leq C_{\kappa}\left(  \dfrac{1}{1+|n|}\right)  ^{\kappa+3}\,,
\end{align*}
and%
\begin{align*}
S_{n/2}^{+}  &  ={\sum_{\left[  n/2\right]  +1}^{\infty}}\dfrac{1}%
{(1+|l|)^{\kappa+1}(1+|n-l|)^{{\kappa+3}}}\\
&  \leq C_{\kappa}\left(  \dfrac{1}{1+|n|}\right)  ^{\kappa+1}\sum_{l\in\mathbb Z}\dfrac{1}{(1+|l|)^{\kappa+3}}  \leq C_{\kappa}\left(  \dfrac{1}{1+|n|}\right)  ^{\kappa+1}\,.
\end{align*}
This shows that $ \hat{F}\in \mathcal{B}%
_{4+2\alpha,\kappa+1}$.
\end{proof}

\medskip

\subsection{Proof of Lemma \ref{lem_S}}
From now on, we assume $\kappa >0.$ 
We recall that:
\[
\zeta_{n}:=\left[  n^{2}+i\mu n\right]  ^{\frac{1}{2}}\,,\quad \forall
\,n\in \mathbb{Z}\,.
\]
In what follows, we use without mention the following properties of $\zeta
_{n}$:
\[
|\zeta_{n}|=|n|\left(  1+\left(  \frac{\mu}{n}\right)  ^{2}\right)  ^{\frac
{1}{4}}\,,\quad \forall \,n\in \mathbb{Z}\setminus \{0\} \,,
\]
and%
\begin{equation}%
\begin{array}
[c]{rcl}%
\xi_{n}:=\mathcal{R}e(\zeta_{n}) & = & \dfrac{|n|}{\sqrt{2}}\left[  \left(
1+\left(  \dfrac{\mu}{n}\right)  ^{2}\right)  ^{\frac{1}{2}}+1\right]
^{\frac{1}{2}}\,,\\
\mathcal{I}m(\zeta_{n}) & = & \dfrac{n}{\sqrt{2}}\left[  \left(  1+\left(
\dfrac{\mu}{n}\right)  ^{2}\right)  ^{\frac{1}{2}}-1\right]  ^{\frac{1}{2}}\,,
\end{array}
\quad \forall \,n\in \mathbb{Z}\setminus \{0\} \,. \label{def_win}%
\end{equation}
We note that $\xi_{n}$ is an increasing function of $|n|$ so that its minimal
value (over $n\in \mathbb{Z}\setminus \{0\}$) is reached for $n=\pm1$ and is
equal to:
\[
\rho_{\mu}:=\dfrac{1}{\sqrt{2}}\left[  \left(  1+\mu^{2}\right)  ^{\frac{1}%
{2}}+1\right]  ^{\frac{1}{2}}\,.
\]
Let%
\[
\alpha_{\mu}:= \dfrac{1}{2}\min(\rho_{\mu}-2,1)\,.
\]
For $\mu>{\mu_{crit}}$ we have $\rho_{\mu}>2$, so that $\alpha_{\mu}>0$. We
choose $\alpha \in(0,\alpha_{\mu})$ from now on. This is the smallness condition 
that is mentioned in Lemma \ref{lem_S}.

\medskip

With the above conventions, we first analyze the equations which determine
$\hat{w}$:

\begin{proposition}
\label{prop_S_w} Given $\hat{F}\in \mathcal{B}_{2\alpha+4,\kappa}$ and
$\hat{w}^{\ast}\in \mathcal{B}_{\kappa+2}^{0}$, the equations:
\begin{align}
w_{n}(r)  &  =\dfrac{\overline{w}_{n}}{r^{\zeta_{n}}}+\int_{r}^{\infty}%
\dfrac{sF_{n}(s)}{2\zeta_{n}}\left(  \dfrac{r}{s}\right)  ^{\zeta_{n}%
}\text{$\, \mathrm{d}s$}+\int_{1}^{r}\dfrac{sF_{n}(s)}{2\zeta_{n}}\left(
\dfrac{s}{r}\right)  ^{\zeta_{n}}\text{$\, \mathrm{d}s$}\,, \label{prf_gamman}%
\\[6pt]
w_{0}(r)  &  =-\int_{r}^{\infty}\dfrac{1}{s}\int_{s}^{\infty}t\,F_{0}%
(t)\text{$\, \mathrm{d}t\, \mathrm{d}s$}\,,
\end{align}
with:
\[
\overline{w}_{n}=w_{n}^{\ast}-\int_{1}^{\infty}\dfrac{sF_{n}(s)}{2\zeta_{n}%
}\left(  \dfrac{1}{s}\right)  ^{\zeta_{n}}\text{$\, \mathrm{d}s\,,$}%
\]
define $C^{2}$ functions. Moreover, we have $\hat{w}\in \mathcal{U}%
_{\alpha+2,\kappa+2}^{2}$, and there exists a constant $C_{\alpha,\mu
}<\infty$, depending only on $\alpha$ and $\mu$, such that
\begin{equation}
\Vert(w_{n})_{n\in \mathbb{Z}}\ ;\  \mathcal{U}_{\alpha+2,\kappa+2}^{2}\Vert \leq
C_{\alpha,\mu}\, \left[  \Vert(F_{n})_{n\in \mathbb{Z}}\ ;\  \mathcal{B}_{4+2\alpha,\kappa
+2}\Vert+\Vert(w_{n}^{\ast})_{n\in \mathbb{Z}}\ ;\  \mathcal{B}_{\kappa+2}%
^{0}\Vert \right]  \,. \label{regw}%
\end{equation}

\end{proposition}

\begin{proof}
We only prove \eqref{regw}; existence and continuity follow in a
straightforward way. We first treat the case $n\neq0$ and then the case $n=0$.
Throughout the proof, we use the shorthand $M_{F}$ for $\Vert(F_{n}%
)_{n\in \mathbb{Z}}\ ;\  \mathcal{B}_{4+2\alpha,\kappa}\Vert$.

\medskip

\textbf{Case $n\neq0$.} We split the
expression defining $w_{n}$ according to (\ref{prf_gamman}):
\[
w_{n}(r)=\dfrac{\overline{w}_{n}}{r^{\zeta_{n}}}+I_{1}^{n}(r)+I_{\infty}%
^{n}(r)\,.
\]
By definition of the norm on the space $\mathcal{B}_{4+2\alpha,\kappa}$, we
have:
\[
|F_{n}(s)|\leq \dfrac{M_{F}}{(1+|n|)^{\kappa}\,s^{4+2\alpha}}\,,\qquad
\forall \,s\geq1\,.
\]
Using that $\xi_{n}\geq \rho_{\mu}>\alpha+2$, and that $c |n| \leq  |\zeta_{n}| \leq C(1+\sqrt{|\mu|})|n|$ and
$|n| \leq \xi_{n} \leq (1+\sqrt{|\mu|})|n|$ for large values of $|n|$, we get that:%
\begin{align}
|I_{1}^{n}(r)|  &  =\left \vert {\int_{1}^{r}}\dfrac{sF_{n}(s)}{2\zeta_{n}%
}\left(  \dfrac{s}{r}\right)  ^{\zeta_{n}}\text{$\, \mathrm{d}s$}\right \vert
\nonumber \\[10pt]
\  &  \leq \dfrac{M_{F}}{(1+|n|)^{\kappa}|\zeta_{n}|}\  \dfrac{1}{r^{\xi_{n}}%
}\ {{\int_{1}^{r}}}s^{1+\xi_{n}-(4+2\alpha)}\text{$\, \mathrm{d}s$%
}\nonumber \\[10pt]
\  &  \leq C_{\alpha,{\mu}}\, \dfrac{M_{F}}{(1+|n|)^{\kappa+2}}\dfrac{\left(
1+r^{\xi_{n}-(2\alpha+2)}\right)  }{r^{\xi_{n}}}\,, \label{In1}%
\end{align}
for all $r\geq1$. Here we used that, by our smallness condition on $\alpha,$ we have $\xi_n - (2\alpha+2) >-1$.
We also have:
\begin{align}
|I_{\infty}^{n}(r)|  &  =\left \vert {\int_{r}^{\infty}}\dfrac{sF_{n}%
(s)}{2\zeta_{n}}\left(  \dfrac{r}{s}\right)  ^{\zeta_{n}}\text{$\,
\mathrm{d}s$}\right \vert \nonumber \\[10pt]
&  \leq \dfrac{M_{F}}{(1+|n|)^{\kappa}|\zeta_{n}|}\ r^{\xi_{n}}\ {{\int
_{r}^{\infty}}}s^{1-\xi_{n}-(2\alpha+4)}\text{$\, \mathrm{d}s$}%
\nonumber \\[10pt]
&  \leq C_{\alpha,{\mu}}\, \dfrac{M_{F}}{(1+|n|)^{\kappa+2}}\dfrac{1}{r^{2\alpha+2}%
}\,, \label{Ininf}%
\end{align}
for all $r\geq1.$  
Using these bounds for $r=1$, we obtain:%
\begin{equation}
|\overline{w}_{n}|\leq|w_{n}^{\ast}|+C_{\alpha,{\mu}}\, \dfrac{M_{F}}{(1+|n|)^{\kappa+2}%
}\,. \label{wnbar}%
\end{equation}
Plugging (\ref{In1}), (\ref{Ininf}) and (\ref{wnbar}) into (\ref{prf_gamman})
and recalling that $\xi_{n}>\alpha+2$ yields:
\[
|w_{n}(r)|\leq C_{\alpha,\mu}\, \dfrac{M_{F}+\Vert \hat{w}^{\ast}\,;\,
\mathcal{B}_{\kappa+2}^{0}\Vert}{(1+|n|)^{\kappa+2}}\dfrac{1}{r^{\alpha+2}%
}\,,\qquad \forall \,r\geq1\,.
\]
Differentiating (\ref{prf_gamman}) with respect to $r$, we obtain:
\[
\partial_{r}w_{n}(r)=-\dfrac{\zeta_{n}\overline{w}_{n}}{r^{\zeta_{n}+1}}%
+\frac{\zeta_{n}}{r}I_{1}^{n}(r)-\frac{\zeta_{n}}{r}I_{\infty}^{n}%
(r)\,,\quad \forall \,r\geq1\,.
\]
To summarize, when we differentiate $w_n$ with respect to
$r$, the decay in $r$ increases by one power, and the decay in $n$ decreases
by one power. This observation allows us to bound $\partial_{r}w_{n}$ in
the indicated function spaces. Finally, since the expression defining $\hat
{w}$ define a solution of \eqref{Fourier}, we plug the bounds on $w_{n}$
and $\partial_{r}w_{n}$ into this equation and get a bound for $\partial
_{rr}w_{n}(r)$. We obtain  that there exists a constant
$C_{\alpha,\mu}$, depending only on $\alpha$ and $\mu$, such that%
\[
\dfrac{r^{2}|\partial_{rr}w_{n}(r)|}{(1+|n|)^{2}}+\dfrac{r|\partial_{r}%
w_{n}(r)|}{(1+|n|)}+|w_{n}(r)|\leq C_{\alpha,{\mu}}\, \dfrac{M_{F}+\Vert
\hat{w}^{\ast}\,;\, \mathcal{B}_{\kappa+2}^{0}\Vert}{(1+|n|)^{\kappa+2}}%
\dfrac{1}{r^{\alpha+2}}\,,\qquad \forall \,r\geq1\,.
\]
We emphasize here that the constant $C_{\alpha,\mu}$ depends on $\alpha$ and $\mu.$
Nevertheless, it is clear from the computations above that, when $\alpha$ is fixed 
and $\mu$ varies in a a compact interval $I \subset \mathbb R,$    this constant
remains uniformly bounded. 

\medskip

\textbf{Case $n=0$.} Proceeding as in the case $n\neq0$, we get the 
bound:%
\[
|w_{0}(r)|\leq M_{F}\int_{r}^{\infty}\dfrac{1}{s}\int_{s}^{\infty}\dfrac
{1}{t^{4+2\alpha}}\, \text{$\, \mathrm{d}t\, \mathrm{d}s$}\leq C_{{\alpha}%
}\dfrac{M_{F}}{r^{2+\alpha}}\,,\quad \forall \,r\geq1\,.
\]
Similarly, one shows
\[
|\partial_{r}w_{0}(r)|\leq \dfrac{M_{F}}{r}\int_{r}^{\infty}\dfrac
{1}{s^{4+2\alpha}}\text{$\, \mathrm{d}s\,$}\leq C_{{\alpha}}\dfrac{M_{F}%
}{r^{3+\alpha}}\,,\quad \forall \,r\geq1\,,
\]
and we again conclude, by recalling the differential
equation satisfied by $w_{0}$ (see \eqref{Fourier} for $n=0$), that:
\[
{r^{2}|\partial_{rr}w_{0}(r)|}+{r|\partial_{r}w_{0}(r)|}+|w_{0}(r)|\leq
C_{\alpha}\, \dfrac{M_{F}+\Vert \hat{w}^{\ast}\,;\, \mathcal{B}_{\kappa+2}%
^{0}\Vert}{r^{2+\alpha}}\, \,,\qquad \forall \,r\geq1\,.
\]
This completes the proof.
\end{proof}

\medskip

We next consider the equation satisfied by $\hat{\gamma}$:

\begin{proposition}
\label{lem_S_gamma} Given $\hat{\phi}\in \mathcal{B}_{\alpha+2,\kappa+2}$ and
$\hat{\gamma}^{\ast}\in \mathcal{B}_{\kappa+4}^{0}$, the equations:
\begin{align}
\gamma_{n}(r)  &  =\dfrac{\overline{\gamma}_{n}}{r^{|n|}}-\int_{r}^{\infty
}\dfrac{s\phi_{n}(s)}{2|n|}\left(  \dfrac{r}{s}\right)  ^{|n|}\text{$\,
\mathrm{d}s$}-\int_{1}^{r}\dfrac{s\phi_{n}(s)}{2|n|}\left(  \dfrac{s}%
{r}\right)  ^{|n|}\text{$\, \mathrm{d}s$}\,,\\[10pt]
\gamma_{0}(r)  &  =\int_{r}^{\infty}\dfrac{1}{s}\int_{s}^{\infty}t\, \phi
_{0}(t)\text{$\, \mathrm{d}t\, \mathrm{d}s$}\,,
\end{align}
with
\[
\overline{\gamma}_{n}=\gamma_{n}^{\ast} + \int_{1}^{\infty}\dfrac{s\phi_{n}%
(s)}{2|n|}\left(  \dfrac{1}{s}\right)  ^{|n|}\text{$\, \mathrm{d}s$}\,,
\]
define $C^{2}$ functions. Moreover, $\hat{\gamma}\in \mathcal{U}_{\alpha
,\kappa+4}^{2}$ and there exists a constant $C_{\alpha}<\infty$,
depending only on $\alpha$, such that
\begin{equation}
\Vert(\gamma_{n})_{n\in \mathbb{Z}}\ ;\  \mathcal{U}_{\alpha,\kappa+4}^{2}%
\Vert \leq C_{\alpha}\left[  \Vert(\phi_{n})_{n\in \mathbb{Z}}%
\ ;\  \mathcal{B}_{\alpha+2,\kappa+2}\Vert+\Vert(\gamma_{n}^{\ast}%
)_{n\in \mathbb{Z}}\ ;\  \mathcal{B}_{\kappa+4}^{0}\Vert \right]  \,.
\label{gammareg}%
\end{equation}

\end{proposition}

The proof is identical to the proof of Proposition \ref{prop_S_w} and is left
to the reader. Lemma \ref{lem_S} is a straightforward consequence of
Proposition \ref{lem_S_gamma} and Proposition \ref{prop_S_w}.

\subsection{Proof of Lemma \ref{lem_Smu}, first item}

Let $I=[\mu_{-},\mu_{+}] \subset (\mu_{crit},\infty)$ and $\alpha$ be given by \eqref{eq_alpha}. 
In particular, we have $\alpha < \min \{ \alpha_{\mu},\mu \in
I\}$ so that, applying the results of the preceding section, it follows that $\mathcal{S}_{\mu}\colon \mathcal{B}_{4+2\alpha,\kappa}%
\times(\mathcal{B}_{\kappa+4}^{0}\times \mathcal{B}_{\kappa+2}^{0}%
)\longrightarrow \mathcal{U}_{\alpha,\kappa+4}^{2}\times \mathcal{U}%
_{\alpha+2,\kappa+2}^{2}$ is a well-defined continuous linear map, for all values of
$\mu \in I$. We now show that the map $\mathcal{S}\colon \mu \mapsto
\mathcal{S}_{\mu}$ is also continuous. This amounts to show that, for arbitrary
$\mu_{0}\in I$ and $\mu \in I,$ there exists a constant $C_{\mu}$ which
converges to zero as $\mu$ converges to $\mu_{0}$, such that, for arbitrary
$(\hat{F},\hat{x}^{\ast})\in \mathcal{B}_{4+2\alpha,\kappa}\times
(\mathcal{B}_{\kappa+4}^{0}\times \mathcal{B}_{\kappa+2}^{0})$:
\[
\Vert \mathcal{S}_{\mu_{0}}(\hat{F},\hat{x}^{\ast})-\mathcal{S}_{{\mu}}(\hat
{F},\hat{x}^{\ast})\,;\, \mathcal{U}_{\alpha,\kappa+4}^{2}\times
\mathcal{U}_{\alpha+2,\kappa+2}^{2}\Vert \leq C_{\mu}\Vert(\hat{F},\hat
{x}^{\ast})\,;\, \mathcal{B}_{4+2\alpha,\kappa}\times(\mathcal{B}%
_{\kappa+4}^{0}\times \mathcal{B}_{\kappa+2}^{0})\Vert \,.
\]
Given $\mu \in I$, we let $(\hat{\gamma}[\mu],\hat{w}[\mu]):=\mathcal{S}%
_{{\mu}}(\hat{F},\hat{x}^{\ast})$. Since $\hat{\gamma}[\mu]$ is obtained
from $\hat{w}[\mu]$ \textit{via} an equation which does not depend on $\mu$,
we can  apply directly Proposition \ref{prop_S_w}, yielding that, for arbitrary
$(\mu,\tilde{\mu})\in I^{2}$:
\[
\Vert \hat{\gamma}[\mu]-\hat{\gamma}[\tilde{\mu}]\,;\, \mathcal{U}_{\alpha
,\kappa+4}^{2}\Vert \leq C_{\alpha}\Vert \hat{w}[\mu]-\hat
{w}[\tilde{\mu}]\,;\, \mathcal{U}_{\alpha+2,\kappa+2}^{2}\Vert \,.
\]
Hence, it suffices to prove that $\hat{w}[\mu]$ is continuous with respect to
$\mu$, in order to obtain the continuity of $\mathcal{S}$.

To show the continuity of $\hat{w}[\mu]$, we first remark that $w_{0}[\mu]$
does not depend $\mu$, so that we only detail the case $n\neq0$. 
Let $n\neq 0,$ we split ${w}_{n}[\mu]$ into three
terms:%
\[
w_{n}[\mu](r)=W_{b}^{n}[\mu](r)+I_{1}^{n}[\mu](r)+I_{\infty}^{n}[\mu](r)\,,
\]
where:
\[
W_{b}^{n}[\mu](r)=\dfrac{\bar{w}_{n}[\mu]}{r^{\zeta_{n}[\mu]}}\,,\quad I_{1}^{n}%
[\mu](r)={\int_{1}^{r}}\dfrac{sF_{n}(s)}{2\zeta_{n}[\mu]}\left(  \dfrac{s}%
{r}\right)  ^{\zeta_{n}[\mu]}\text{$\, \mathrm{d}s$}\,,\quad I_{\infty}^{n}%
[\mu](r)=\int_{r}^{\infty}\dfrac{sF_{n}(s)}{2\zeta_{n}[\mu]}\left(  \dfrac{r}%
{s}\right)  ^{\zeta_{n}[\mu]}\text{$\, \mathrm{d}s$}\,.
\]
We recall that
\begin{align}
\partial_{r}w_{n}[\mu](r)  &  =-\dfrac{\zeta_{n}[\mu]}{r}W_{b}^{n}[\mu
](r)-\dfrac{\zeta_{n}[\mu]}{r}I_{1}^{n}[\mu](r)+\dfrac{\zeta_{n}[\mu]}{r}I_{\infty}%
^{n}[\mu](r)\,,\label{eq_derw1}\\[10pt]
\partial_{rr}w_{n}[\mu](r)  &  =F_{n}(r)-\dfrac{\partial_{r}w_{n}[\mu](r)}%
{r} +  \dfrac{(n^{2}+i\mu n)}{r^2}w_{n}[\mu](r)\,. \label{eq_derw2}%
\end{align}
Note also that $\zeta_{n}[\mu]=(n^{2}+i\mu n)^{1/2}$ is a continuous function of $\mu$
uniformly in $n$. Indeed, since the square root is analytic in a neighborhood
of $1$, we have for sufficiently large $n$ (uniformly in $\mu \in I$):
\[
|\zeta_{n}[\mu]-\zeta_{n}[\tilde{\mu}]|=|n|\left \vert \left(  1+\dfrac{i\mu
}{n}\right)  ^{\frac{1}{2}}-\left(  1+\dfrac{i\tilde{\mu}}{n}\right)
^{\frac{1}{2}}\right \vert \leq C|\mu-\tilde{\mu}|\,.
\]
We also have the bound $|\zeta_{n}[\mu]|\leq c|n|,$ with $c$ independent of
$\mu \in I$. Introducing these bounds into \eqref{eq_derw1}-\eqref{eq_derw2}
shows that the continuity of $\hat{w}[\mu]$ follows from the
continuity of $(\hat{W}_{b}[\mu],\hat{I}_{1}[\mu],\hat
{I}_{\infty}[\mu])$ in $\mathcal{B}_{\alpha+2,\kappa+2}$.  For consistency, the three sequences,
which are only defined for $n\neq0,$ are completed by $0$ for $n=0$.

To begin with, we consider the continuity of $\mu \mapsto \hat{I}_{1}[\mu]$. Let
$(\mu,\tilde{\mu})\in I^{2}$, and assume that $|\zeta_{n}[\mu]-\zeta
_{n}[\tilde{\mu}]|<\alpha/2$, uniformly in $n$. We have:
\[
I_{1}^{n}[\mu]-I_{1}^{n}[\tilde{\mu}]=J_{1}(r)+J_{2}(r)\,,
\]
where:
\begin{align*}
J_{1}(r)  &  =\int_{1}^{r}\dfrac{sF_{n}(s)}{2\zeta_{n}[\mu]}\left[
\dfrac{\zeta_{n}[\tilde{\mu}]-\zeta_{n}[\mu]}{\zeta_{n}[\tilde{\mu}]}\right]
\left(  \dfrac{s}{r}\right)  ^{\zeta_{n}[\mu]}\, \mathrm{d}\text{$s$%
}\,,\\[6pt]
J_{2}(r)  &  =\int_{1}^{r}\dfrac{sF_{n}(s)}{2\zeta_{n}[\tilde{\mu}]}\left[
1-\left(  \dfrac{s}{r}\right)  ^{\zeta_{n}[\mu]-\zeta_{n}[\tilde{\mu}%
]}\right]  \left(  \dfrac{s}{r}\right)  ^{\zeta_{n}[\tilde{\mu}]}\,
\mathrm{d}\text{$s$}\,.
\end{align*}
We have, uniformly in $n$:
\begin{align*}
\left \vert \dfrac{\zeta_{n}[\tilde{\mu}]-\zeta_{n}[\mu]}{\zeta_{n}[\tilde{\mu
}]}\right \vert  &  \leq c|\zeta_{n}[\mu]-\zeta_{n}[\tilde{\mu}%
]|=o(1)\,,\\[10pt]
\left \vert 1-\left(  \dfrac{s}{r}\right)  ^{\zeta_{n}[\mu]-\zeta_{n}%
[\tilde{\mu}]}\right \vert  &  \leq c\left \vert \zeta_{n}[\mu]-\zeta_{n}%
[\tilde{\mu}]\right \vert \ln(r)r^{|\zeta_{n}[\mu]-\zeta_{n}[\tilde{\mu}%
]|}=o(1)\ln(r)r^{\alpha/2}\,.
\end{align*}
We introduce these uniform bounds in $J_1$ and $J_2$ and redo the computations in the
proof of {Proposition \ref{prop_S_w}} (see \eqref{In1}). We get:
\[
|J_{1}(r)|\leq \dfrac{\Vert \hat{F};\mathcal{B}_{4+2\alpha}\Vert
}{(1+|n|^{\kappa+2})}\  \dfrac{o(1)}{r^{2+2\alpha}}\,,\quad|J_{2}%
(r)|\leq \dfrac{\Vert \hat{F};\mathcal{B}_{4+2\alpha}\Vert}{(1+|n|^{\kappa
+2})}\  \dfrac{o(1)\ln(r)}{r^{2+3\alpha/2}}\,,
\]
where the term $o(1)$ denotes a constant converging to $0$ when $\mu-\tilde{\mu} \rightarrow0$,
uniformly in $n\,$. Finally, we have that, for all $n\in \mathbb{Z}\setminus \{0\}$:
\begin{equation}
|I_{1}^{n}[\mu]-I_{1}^{n}[\tilde{\mu}]|\leq \dfrac{\Vert \hat{F};\mathcal{B}%
_{4+2\alpha}\Vert}{(1+|n|^{\kappa+2})}\  \dfrac{o(1)}{r^{2+\alpha}}\,, \label{eq_diffI1}%
\end{equation}

We now prove the continuity of $\mu \mapsto \hat{I}_{\infty}[\mu]$. For any
$(\mu,\tilde{\mu})\in I$ and $n\neq0$, we perform a similar splitting:
\[
I_{\infty}^{n}[\mu]-I_{\infty}^{n}[\tilde{\mu}]=J_{1}(r)+J_{2}(r)\,,
\]
where:
\begin{align*}
J_{1}(r)  &  =\int_{r}^{\infty}\dfrac{sF_{n}(s)}{2\zeta_{n}[\mu]}\left[
\dfrac{\zeta_{n}[\tilde{\mu}]-\zeta_{n}[\mu]}{\zeta_{n}[\tilde{\mu}]}\right]
\left(  \dfrac{r}{s}\right)  ^{\zeta_{n}[\mu]}\mathrm{d}\text{$s$}\,,\\[6pt]
J_{2}(r)  &  =\int_{r}^{\infty}\dfrac{sF_{n}(s)}{2\zeta_{n}[\tilde{\mu}%
]}\left[  1-\left(  \dfrac{r}{s}\right)  ^{\zeta_{n}[\mu]-\zeta_{n}[\tilde
{\mu}]}\right]  \left(  \dfrac{r}{s}\right)  ^{\zeta_{n}[\tilde{\mu}%
]}\mathrm{d}\text{$s$}\,.
\end{align*}
As in the preceding bound we have, uniformly in $n$:
\begin{align*}
\left \vert \dfrac{\zeta_{n}[\tilde{\mu}]-\zeta_{n}[\mu]}{\zeta_{n}[\tilde{\mu
}]}\right \vert  &  \leq c|\zeta_{n}[\mu]-\zeta_{n}[\tilde{\mu}%
]|=o(1)\,,\\[10pt]
\left \vert 1-\left(  \dfrac{s}{r}\right)  ^{\zeta_{n}[\mu]-\zeta_{n}%
[\tilde{\mu}]}\right \vert  &  \leq o(1)\ln \left(  \dfrac{s}{r}\right)  \left(
\dfrac{s}{r}\right)  ^{\alpha/2}\,,
\end{align*}
where we have used that  $|\zeta_{n}[\mu]-\zeta_{n}[\tilde{\mu}%
]|\leq \alpha/2$. We can therefore redo the computations in the proof of
{Proposition \ref{prop_S_w}} (see \eqref{Ininf}). This yields, for all $n\in\mathbb Z \setminus\{0\}$:
\[
|J_{1}(r)| + |J_2(r)| \leq \dfrac{\Vert \hat{F};\mathcal{B}_{4+2\alpha,\kappa}\Vert
}{(1+|n|^{\kappa+2})}\  \dfrac{o(1)}{r^{2+\alpha}}\,.
\]
As in the preceding estimate we conclude that:
\[
\Vert \hat{I}_{\infty}[\mu]-\hat{I}_{\infty}[\tilde{\mu}];\mathcal{B}_{\alpha
_{-}+2,\kappa+2}\Vert=o(1)\Vert \hat{F};\mathcal{B}_{4+2\alpha,\kappa}\Vert \,.
\]
Finally, we prove the continuity of $\mu \mapsto \hat{W}_b[\mu]$:
\begin{eqnarray*}
\left| W^n_b[\mu](r) - W^n_b[\tilde{\mu}](r) \right| &=&  \left| \dfrac{\bar{w}_n[\mu] - \bar{w}_n[\tilde{\mu}]}{r^{\zeta_n[\mu]}}  +  \dfrac{\bar{w}_n[\tilde{\mu}]}{r^{\zeta_n[\mu_-]}} \left( \dfrac{1}{r^{\zeta_n[\mu]- \zeta_n[\mu_-]}} - \dfrac{1}{r^{\zeta_n[\tilde{\mu}]- \zeta_n[\mu_-]}}\right) \right| \,, \\[10pt]
&\leq & 
 \left| \dfrac{\bar{w}_n[\mu] - \bar{w}_n[\tilde{\mu}]}{r^{\zeta_n[\mu]}} \right| + \left| \dfrac{\bar{w}_n[\tilde{\mu}]}{r^{\zeta_n[\mu_-]}} \left( \dfrac{1}{r^{\zeta_n[\mu]- \zeta_n[\mu_-]}} - \dfrac{1}{r^{\zeta_n[\tilde{\mu}]- \zeta_n[\mu_-]}}\right) \right| \,, \\[10pt]
 &\leq &
 \dfrac{|\bar{w}_n[\mu] - \bar{w}_n[\tilde{\mu}]|}{r^{2+2\alpha}} + \dfrac{|\bar{w}_n[\tilde{\mu}]|}{r^{2+\alpha}} \dfrac{1}{r^{\alpha}} \left|\left( \dfrac{1}{r^{\zeta_n[\mu]- \zeta_n[\mu_-]}} - \dfrac{1}{r^{\zeta_n[\tilde{\mu}]- \zeta_n[\mu_-]}}\right) \right|\,,
\end{eqnarray*}
where we have used that $\mathcal{R}e(\zeta_n[\mu]) \geq \mathcal{R}e(\zeta_n[\mu_-]) > 2+ 2\alpha.$
At this point we note that  the bound which we obtained above for $I_{\infty}^n$ in $r=1$ yields: 
$$
|\bar{w}_n[\mu] - \bar{w}_n[\tilde{\mu}]| = \dfrac{o(1)}{(1+|n|)^{\kappa+2}}\Vert \hat{F};\mathcal{B}_{4+2\alpha,\kappa}\Vert\,.
$$ 
As $\mu \mapsto \zeta_n[\mu]$ is continuous in $\mu$ (uniformly in $n$) and  satisfies $\mathcal{R}e(\zeta_n[\mu])  \geq \mathcal{R}e(\zeta_n[\mu_-]),$
for all $\mu \in I,$ we also have:
$$
\left\| \dfrac{1}{r^{\alpha}} \left( \dfrac{1}{r^{\zeta_n[\mu]- \zeta_n[\mu_-]}} - \dfrac{1}{r^{\zeta_n[\tilde{\mu}]- \zeta_n[\mu_-]}}\right) \, ; \, L^{\infty}(1,\infty) \right\|  = o(1)\,,
$$
where $o(1)$ is uniform in $n.$
By combination, this yields, for all $n \in \mathbb Z \setminus \{0\}$:
$$
\left|
W^n_b[\mu](r) - W^n_b[\tilde{\mu}](r) 
\right| \leq \dfrac{o(1)}{(1+|n|^{\kappa+2}) r^{2+\alpha}} \left[ \Vert \hat{F};\mathcal{B}_{4+2\alpha,\kappa}\Vert + \Vert \hat{w}^*;\mathcal{B}^0_{\kappa+2}\Vert \right]\,.
$$
This completes the proof of the first item in Lemma \ref{lem_Smu}.

\subsection{Proof of Lemma \ref{lem_Smu}, second item}

In this paragraph, we prove that the map $\mathcal{T}_{1}$ is one-to-one and onto with a
continuous inverse. Given $\hat{x}^{\ast}=(\hat{\gamma}^{\ast},\hat{w}^{\ast
})\in \mathcal{B}_{\kappa+4}^{0}\times \mathcal{B}_{\kappa+2}^{0}$, we set
$(\hat{\gamma},\hat{w})=\mathcal{S}_{\mu}(0,\hat{x}^*)$. A straightforward
computation shows:
\[
w_{n}(r)=\dfrac{w_{n}^{\ast}}{r^{\zeta_{n}}}\qquad \gamma_{n}(r)=\dfrac
{\bar{\gamma}_{n}}{r^{|n|}} + \int_{r}^{\infty}\dfrac{sw_{n}(s)}{2|n|}\left(
\dfrac{r}{s}\right)  ^{|n|}\text{$\, \mathrm{d}s$} + \int_{1}^{r}\dfrac
{sw_{n}(s)}{2|n|}\left(  \dfrac{s}{r}\right)  ^{|n|}\text{$\, \mathrm{d}s$%
}\,,\quad \forall \,r\geq1\,,
\]
where:
\[
\bar{\gamma}_{n}=\gamma_{n}^{\ast}-\int_{1}^{\infty}\dfrac{sw_{n}(s)}%
{2|n|}\left(  \dfrac{1}{s}\right)  ^{|n|}\text{$\, \mathrm{d}s$}\,.
\]
Therefore, we have, for all $n\in \mathbb{Z}\setminus \{0\}$:
\[
\gamma_{n}(1)=\gamma_{n}^{\ast}\,,
\]
together with:%
\[
\partial_{r}\gamma_{n}(1)=-|n|\gamma_{n}^{\ast} + {\int_{1}^{\infty}}%
sw_{n}(s)\left(  \dfrac{1}{s}\right)  ^{|n|}\text{$\, \mathrm{d}s$}%
=-|n|\gamma_{n}^{\ast} - \dfrac{w_{n}^{\ast}}{2-|n|-\zeta_{n}}\,,
\]
so that $(\hat{v}_r^{\ast},\hat{v}_{\theta}^{\ast})=\mathcal{T}_{1}(\hat{x}^{\ast})$ satisfies:

\begin{itemize}
\item $v_{r,0}^{\ast}=v_{\theta,0}^{\ast}=0$,

\item $v_{r,n}^{\ast}=in\gamma_{n}^{\ast}$, and $v_{\theta,n}^{\ast}= -\partial_r \gamma_n(1) =  |n|\gamma
_{n}^{\ast} + \dfrac{w_{n}^{\ast}}{2-|n|-\zeta_{n}}$, for all $n\in
\mathbb{Z}\setminus \{0\}$.
\end{itemize}

This shows that the map $\mathcal{T}_{1}$ is  one-to-one and onto. Indeed, the inverse map
is given by:
\[
\mathcal{T}_{1}^{-1}[\hat{v}_{r}^{\prime},\hat{v}_{\theta}^{\prime})]=(\hat{\gamma}%
^{\prime},\hat{w}^{\prime})\,,
\]
where:

\begin{itemize}
\item $\gamma_{0}^{\prime}=w_{0}^{\prime}=0,$

\item $\gamma_{n}^{\prime}=\dfrac{v_{r,n}^{\prime}}{in}$, and $w_{n}^{\prime
}=  (2-|n|-\zeta_{n})\left(  v_{\theta,n}^{\prime} - \dfrac{|n|v_{r,n}^{\prime}}%
{in}\right)  $ for all $n\in \mathbb{Z}\setminus \{0\}$.
\end{itemize}
It is therefore clear that $\mathcal{T}_{1}^{-1}\in \mathcal{L}_{c}%
(\mathcal{B}_{\kappa+3}^{0}\times \mathcal{B}_{\kappa+3}^{0}\ ;\  \mathcal{B}%
_{\kappa+4}^{0}\times \mathcal{B}_{\kappa+2}^{0})$. This completes the proof.

\end{document}